\documentclass[12pt]{amsart}
\usepackage{amssymb,amsthm, amsmath}
\usepackage{eucal,mathrsfs}
\usepackage{color}
\usepackage[all]{xy}

\newtheoremstyle{mythm}%
  {\topskip}%        above
  {\topskip}%        below
  {\itshape}%       body font
  {}%               indent amount
  {\bfseries}%      head font
  {\\}%               punctuation after head
  {\parindent}%       after head
  {\thmname{#1}\thmnumber{ #2}\thmnote{ \textit{#3}}}%     head spec
\newtheoremstyle{mydef}%
  {\topskip}%        above
  {\topskip}%        below
  {}%       body font
  {}%               indent amount
  {\bfseries}%      head font
  {\\}%               punctuation after head
  {\parindent}%       after head
  {\thmname{#1}\thmnumber{ #2}\thmnote{ #3}}%     head spec

\theoremstyle{mythm}
\newtheorem{lem}{Lemma}[section]
\newtheorem{thm}[lem]{Theorem}
\newtheorem*{thm*}{Theorem}

\newtheorem{cor}[lem]{Corollary}
\newtheorem*{cor*}{Corollary}
\newtheorem{prop}[lem]{Proposition}

\newtheorem{hyp}[lem]{Hypothesis}
\theoremstyle{mydef}
\newtheorem{dfn}[lem]{Definition}
\newtheorem{rem}[lem]{Remark}
\newtheorem{notn}[lem]{Notation}
\newtheorem{ex}[lem]{Example}

\newcommand{\mbb}[1]{\mathbb #1}
\newcommand{\mf}[1]{\mathfrak #1}
\newcommand{\mc}[1]{\mathcal #1}

\newcommand{\oper}[1]{\operatorname{#1}}
\newcommand{\m}[1]{#1}
\newcommand{\md}[1]{#1}

\newcommand{\dra}{\dashrightarrow}

\newcommand{\op}{\oper{op}}

\newcommand{\Br}{\oper{Br}}
\newcommand{\per}{\oper{per}}

\newcommand{\ind}{\oper{ind}}
\newcommand{\lcm}{\oper{lcm}}
\renewcommand{\deg}{\oper{deg}}
\newcommand{\ram}{\oper{ram}}
\newcommand{\Gal}{\oper{Gal}}

\newcommand{\Hom}{\oper{Hom}}

\newcommand{\cha}{\oper{char}}

\newcommand{\Spec}{\oper{Spec}}
\newcommand{\GL}{\oper{GL}}
\newcommand{\Mat}{\oper{Mat}}

\newcommand{\Orth}{\oper{O}}
\newcommand{\SOrth}{\oper{SO}}
\newcommand{\SB}{\oper{SB}}
\newcommand{\cd}{\oper{cd}}
\newcommand{\td}{\oper{tr.deg.}}
\newcommand{\ur}{\oper{ur}}

\newcommand{\Q}{\mbb Q}
\newcommand{\Z}{\mbb Z}
\newcommand{\N}{\mbb N}
\newcommand{\C}{\mbb C}
\newcommand{\PP}{\mbb P}
\newcommand{\A}{\mbb A}
\newcommand{\F}{\mbb F}

\newcommand{\ov}{\overline}
\newcommand{\til}{\widetilde}
\newcommand{\wh}{\widehat}

\usepackage[OT2,T1]{fontenc}

\def\<{\left<}
\def\>{\right>}

\title[Applications of patching]{Applications of patching to
quadratic forms and central simple algebras}

\author{David Harbater}
\author{Julia Hartmann}
\author{Daniel Krashen}

\begin{document}

\thanks{The first author was supported in part by NSF Grant DMS-0500118}

%-----------------------------------------------------------------------------
\begin{abstract}
This paper provides applications of patching to quadratic forms
and central simple algebras over function fields of curves over
henselian valued fields. 
In particular, we use a patching approach to reprove and generalize a recent result 
of Parimala and Suresh on the $u$-invariant of $p$-adic function fields, $p \ne 2$.  
The strategy relies on a local-global principle for homogeneous spaces for rational 
algebraic groups, combined with local computations.
\end{abstract}

\maketitle

%%%%%%%%%%%%%%%%%%%%%%%%%%%%%%%%%%%%%%%%%%%%%%%%%%%%%%%%%%%%%%%%%%%%%%%%%%%%%%%%
%Section 1
\section{Introduction} \label{intro}
%%%%%%%%%%%%%%%%%%%%%%%%%%%%%%%%%%%%%%%%%%%%%%%%%%%%%%%%%%%%%%%%%%%%%%%%%%%%%%%%

A longstanding open problem in the theory of quadratic forms is to find a general method for evaluating the $u$-invariant of fields.  
To date, though, the $u$-invariant has been computed only in quite restricted situations.  In this paper we prove a general result 
that provides the $u$-invariant of function fields of curves for a variety of open cases, as well as implying known results in a 
unified way.  Most notably, we obtain a new proof of the recent result of Parimala and Suresh (\cite{PaSu}) on the $u$-invariant 
of nondyadic $p$-adic function fields. Our approach also yields evidence for the expected growth of the $u$-invariant, for example upon 
field extensions.

The method used here is quite different from that of \cite{PaSu} and other works 
on this topic, and is not cohomological.  The results stem from a 
local-global principle for the existence of points on certain homogeneous varieties, which yields a Hasse-Minkowski type statement 
for quadratic forms over function fields of curves.  

Our proofs rely on ideas from \textit{patching}, a method that has been 
used in the past to prove many results about Galois theory (see e.g.\
\cite{Har:MSRI}).  In \cite{HH:FP}, the first two authors extended
patching to structures over fields rather than over rings, to 
make the method more amenable to other applications.  This approach 
shows that giving an algebraic structure over certain function fields is 
equivalent to giving the structure over a suitable collection of 
overfields.  As in earlier forms of patching, a key step is to prove a matrix factorization result.  We use these ideas here, 
especially in the 
proof of our local-global principle.

In addition, we show how the same local-global principle can be used to 
obtain results about the period-index problem for central simple 
algebras. In particular, we give a new proof of a recent result of 
Lieblich (\cite{Lie:PI}) on function fields of curves over henselian 
rings.  It has been understood that there is a connection between 
results concerning $u$-invariants and the 
period-index problem for central simple algebras, and it is interesting 
to see how similar our proofs are in these two situations.

Below, we summarize the main results on quadratic forms and central simple
algebras (which can be found in Sections~\ref{quadratic} and~\ref{csa}).
%--------------------------------------------------------------------------------------------------------------------------------------
\subsection{Results on quadratic forms}
We begin by recalling Kaplansky's definition of the $u$-invariant 
(some references use a modified definition due to Elman and Lam which agrees
with this for nonreal fields, see e.g.\ \cite{pfister}, p.~114).

\begin{dfn} \label{uinv}
Let $k$ be a field. The \textit{u-invariant} of $k$, denoted by $u(k)$,
is the maximal dimension of anisotropic quadratic forms over $k$ 
(or $\infty$, if
such dimensions are arbitrarily large).
\end{dfn}

The $u$-invariant and the possible values it can take for a fixed or varying field
has been a major object of study in the theory of quadratic forms. (Note that it is a positive integer if it is finite.)  There are many
open problems concerning this number; see for example, \cite{Lam}, Section~XIII.6. On 
the other hand there has been a lot of recent progress, most notably in the computation of the 
$u$-invariant of function fields of non-dyadic $p$-adic curves due to Parimala and Suresh 
(see below). 

It is generally expected that the $u$-invariant of field extensions should grow along with the cohomological dimension. In particular, for
``reasonable''  fields, one expects that finite extensions have the same
$u$-invariant, and that the $u$-invariant should double upon a finitely
generated field extension of transcendence degree one. To formalize our
discussion towards these expectations, we make the following definition:

\begin{dfn}\label{suinv}
Let $k$ be a field. 
The \textit{strong u-invariant} of $k$, denoted by $u_s(k)$, is the smallest real number $n$ such that 
\begin{enumerate}
\item[-] 
every finite field extension $E/k$ satisfies $u(E) \le n$, and
\item[-]
every finitely
generated field extension $E/k$ of transcendence degree one satisfies $u(E) \leq 2n$.
\end{enumerate}
If these $u$-invariants are arbitrarily large we say that
$u_s(k)=\infty$.
\end{dfn}

Thus $u_s(k) \le n$ if and only if every finitely generated field extension $E/k$ of transcendence degree $\ell \le 1$ satisfies 
$u(E) \le 2^\ell n$.
Since the $u$-invariant, if finite, is a positive integer, it follows that $u_s(k)$ is at least $1$ and lies in ${\frac 12} \N$. 

Concerning quadratic forms, our main result is:

\begin{thm*}[(Theorem~\ref{main})]
Let $K$ be a complete discretely valued field whose residue field $k$ has characteristic unequal to $2$. Then $u_s(K) = 2 u_s(k)$.
\end{thm*}

More generally we show that this holds for excellent henselian discrete valuation rings (Corollary~\ref{hensel}).  
As a consequence of these results, in many cases we are able to obtain exact values of the $u$-invariant and strong $u$-invariant, 
not just upper bounds.

By definition, a $C_d$-field has $u$-invariant at most $2^d$. Using this, we deduce from our main theorem that
if $T$ is a complete (or excellent henselian) discrete valuation ring whose residue field is a $C_d$-field of odd characteristic, 
every function 
field $F$ of a regular $T$-curve satisfies $u(F)\leq 2^{d+2}$ (see Corollary~\ref{cor_mlocal}(\ref{mlocalineq}), which is more general).  
As a special case, we
obtain the recent theorem of Parimala and Suresh (\cite{PaSu}, Theorem~4.6; Corollary~\ref{cor_pasu} below): 
A function field in one variable over a 
non-dyadic $p$-adic field has $u$-invariant $8$.  
Our result also applies to function fields over the algebraic closure of ${\mathbb Q}$ in a non-dyadic $p$-adic field.

Applying induction to our main theorem we obtain that
the $u$-invariant of an $m$-local field with algebraically closed (respectively, finite) residue field of characteristic unequal 
to $2$ is $2^{m+1}$ (resp.,\ $2^{m+2}$); see Corollary~\ref{finac}.  
For example, the $u$-invariant of a one-variable function field over $\Q_p((t))$ is $16$, for $p$ odd.
As another application, let $k$ be a function field of transcendence degree $d$ over an algebraically closed field of characteristic 
unequal to $2$.  Then the $u$-invariant of the function field of a $K$-curve is at most $2^{d+m+1}$, for any $m$-local field $K$ with
 residue field~$k$; see after Corollary~\ref{cor_mlocal}.

In addition to these, we obtain similar results for other classes of
fields which naturally occur in the context of patching, described at the end of Section~\ref{quadratic}.  More specifically, suppose 
that $T$ is a complete discrete valuation ring with uniformizer $t$ and residue field $k$ of
characteristic unequal to $2$.  If $F$ is the fraction field of $T[[x]]$ or of the $t$-adic completion of $T[x]$, then 
$u(F) \le 4 u_s(k)$, with equality if 
$u(k)=u_s(k)$; e.g.\ if
$k$ is a $C_d$-field having $u$-invariant~$2^d$. In particular, if $k$ is algebraically closed, then $u(F)$ is equal to $4$; 
and $u(F)$ equals~$8$ if $k$ is finite (example cases of the latter include $k((x,y))$ and the fraction field of $\Z_p[[x]]$ 
with $p$ odd).

\subsection{Results on central simple algebras}

Given a field $k$, recall that the \textit{period} (or \textit{exponent}) of a central simple $k$-algebra $A$ is the order of 
the class of $A$ in the Brauer group of $k$; and the \textit{index} of $A$ is the degree of the division algebra $D$ that lies 
in the class of $A$ (i.e.\ such that $A$ is a matrix ring over $D$).  The period and index always have the same prime factors, 
and the period always divides the index (\cite{Pie}, Proposition~14.4(b)(ii)).  The \textit{period-index problem} asks whether 
all central simple algebras $A$ over a given field $k$ satisfy $\ind(A)\,|\,\per(A)^d$ for some fixed exponent $d$ depending only on $k$.
In analogy with the notion of the strong $u$-invariant (Definition~\ref{suinv}), we make the following definition (extending that of 
Lieblich; see \cite{Lie:PI}, Definition~1.1):

\begin{dfn}
Let $k$ be a field.  The \textit{Brauer dimension} of $k$ (\textit{away from a prime~$p$}) is defined to be $0$ if $k$ is separably 
closed (resp.\ separably closed away from $p$, i.e.\ the absolute Galois group of $k$ is a pro-$p$ group).  Otherwise, it is the smallest positive 
integer $d$ such that 
\begin{enumerate}
\item[-] 
for every finite field extension $E/k$ and every central simple $E$-algebra $A$ 
(resp.\ with $p {\not |}\per(A)$), we have $\ind(A) | \per(A)^{d-1}$; and
\item[-]
for every finitely generated field extension $E/k$ of
transcendence degree one 
and every central simple $E$-algebra $A$ 
(resp.\ with $p {\not |}\per(A)$), we have $\ind(A) | \per(A)^d$.
\end{enumerate}
If no such number $d$ exists, we say that the Brauer dimension is $\infty$.
\end{dfn}

Again, we can summarize this by saying that the Brauer dimension of $k$ is at most $d$ if for every finitely generated field extension 
$E/k$ of transcendence degree $\ell \le 1$ and
every central simple $E$-algebra $A$ 
(resp.\ with $p {\not |}\per(A)$), we have $\ind(A) | \per(A)^{d+\ell - 1}$.

As with the $u$-invariant, it is expected that this invariant should grow in
parallel to the cohomological dimension. In particular, one expects that it
should increase by one upon a finitely generated field extension of
transcendence degree one.  Early results in this direction were obtained by Saltman in \cite{Sal:DA} and~\cite{Sal:DAC}, 
including the fact that $\operatorname{ind}\mid \operatorname{per}^2$ for $p$-adic curves, along with a general mechanism to 
relate the Brauer dimension of curves over discretely valued fields to that of curves over the residue field.  (See also \cite{For}.)  Along these lines, 
in Section~\ref{csa} we give an alternative proof of a result that was recently shown by Lieblich in the case $d>0$ (\cite{Lie:PI}, Theorem~5.3):

\begin{thm*}[(Theorem~\ref{main_csa})]
Let $K$ be a complete discretely valued field whose residue field $k$ 
has characteristic $0$ (resp.\ characteristic $p > 0$).
If $k$ has Brauer dimension $d \ge 0$ (resp.\ away from $p$) then
$K$ has Brauer dimension at most $d+1$ (resp.\ away from $p$).
\end{thm*}

More generally, as in \cite{Lie:PI}, we show a version of this result for excellent henselian rings.  As an application of the above 
theorem, since the Brauer dimension of a finite field is $1$, it follows that the Brauer dimension of a $p$-adic field is at most $2$, 
and that of $\Q_p((t))$ is at most~$3$.  As another application, let $k$ be the function field of a curve over a separably closed field. 
Then the Brauer dimension of $k$ is $1$ by \cite{deJ}.  So
$\ind(\alpha) = \per(\alpha)$ for all $\alpha$ in the Brauer group of $k((t))$ with $\cha(k)$ not dividing the period.  Similarly, 
$\ind(\alpha)$ divides $\per(\alpha)^2$ for all $\alpha$ in the Brauer group of $k((t))(x)$ of period not divisible by $\cha(k)$.

In analogy to the results on the $u$-invariant, we also obtain statements for fields that arise from patching; see 
Corollary~\ref{main_csa_patches}.  In particular, let $T$ be a complete discrete valuation ring with uniformizer $t$ and residue field 
$k$ of
characteristic $0$ (resp.\ characteristic $p > 0$) and Brauer dimension $d$ (resp.\ away from $p$).  If $F$ is the fraction field of 
$T[[x]]$ or of the $t$-adic completion of $T[x]$, then $\ind(\alpha) | \per(\alpha)^{d+2}$
for all $\alpha \in \Br(F)$ with period not divisible by $\cha(k)$.  
Moreover $\ind(\alpha)=\per(\alpha)$ if $F$ is $k((x,t))$ or the fraction field of $k[x][[t]]$ where $k$ is separably closed, or if $F$ 
is the fraction field of $\Z_p^{\ur}[[x]]$, provided that the residue characteristic does not divide $\per(\alpha)$.

\subsection{Organization of the manuscript}
The organization of the manuscript is as follows. Section~\ref{decomp} is concerned with a 
decomposition of vectors. It is fairly technical and may be skipped upon a first reading.  Section~\ref{rational factorization} shows
 how this decomposition in vector spaces
can be used to obtain a multiplicative decomposition (i.e.\ factorization) in rational linear
algebraic groups (Theorem~\ref{general factorization}).  The main result of the section, the local-global principle for homogeneous 
spaces (Theorem~\ref{torsor_injective}), is a rather direct consequence. It is the key ingredient for proving the upper bounds in the 
later results. 
In Sections~\ref{quadratic} and~\ref{csa}, local computations combined 
with Theorem~\ref{torsor_injective} yield the main results about quadratic forms 
and central simple algebras, respectively.

\medskip

\noindent{\bf Acknowledgment.} The authors thank Karim Johannes Becher, Jean-Louis
Colliot-Th\'el\`ene, R.~Parimala, Jakob Stix and V.~Suresh for their comments on this manuscript.

%%%%%%%%%%%%%%%%%%%%%%%%%%%%%%%%%%%%%%%%%%%%%%%%%%%%%%%%%%%%%%%%%%%%%%%%%%%%%%%%
%Section 2
\section{Decomposition of vectors} \label{decomp}
%%%%%%%%%%%%%%%%%%%%%%%%%%%%%%%%%%%%%%%%%%%%%%%%%%%%%%%%%%%%%%%%%%%%%%%%%%%%%%%%

The goal of this section is to prove a decomposition theorem 
(Theorem~\ref{abstract small factorization})
that will be used in the next section to obtain factorization results and a local-global principle for rational linear algebraic groups.  This strategy parallels that of \cite {HH:FP}, which concerned the group $\GL_n$.

Throughout this section we let $F_0$ be the fraction field of a 
complete discrete valuation ring $\wh R_0$ with uniformizer $t$, and 
we let $|\ |$ be a norm on $F_0$ induced by the $t$-adic valuation ---
i.e. $|a| = \alpha^{-v(a)}$ for a real number $\alpha > 1$.
This norm extends uniquely to a norm on a fixed algebraic closure $\bar F_0$ of $F_0$ (again denoted by $|\ |$).
If $E \subseteq \bar F_0$ is a field extension of $F_0$ and $\md V$ is a 
finite dimensional vector space over $E$ with basis $b_1, \ldots, b_n$, we 
define a norm on $\md V$ by setting $|\sum a_i b_i| = \max\{|a_i|\}$. Since $\md
V$ is finite dimensional, it is complete with respect to this metric if $E$ is finite over $F_0$.  We will commonly identify such a vector space $\md V$ with the points of the
affine space $\mbb A^n_{F_0}(E)$ and consequently talk about the norm of such points as well.  

For $n \ge 0$, the $t$-adic topology on $\A^n_{F_0}(\bar F_0)$ is finer than the Zariski topology.  This is because a basic open set in the Zariski topology is defined by the non-vanishing of a polynomial $f \in F_0[x_1,\dots,x_n]$, and because such an $f$ is continuous in the $t$-adic topology.

Now fix $n$, let $A = F_0[x_1, \ldots, x_{2n}]$ be the coordinate ring of
$\mbb A_{F_0}^{2n}$, and let $\wh A = F_0[[x_1, \ldots, x_{2n}]]$ be the completion at the maximal ideal $\mf m_0$ at the origin.  
Also let $A_0$ be the localization of $A$ at $\mf m_0$; thus $A_0 \subset \wh A$.  For short, we write $x$ for $(x_1,\dots,x_{2n})$.
Given a $2n$-tuple  $\nu = (\nu_1, \ldots, \nu_{2n}) \in \N^{\,2n}$ of nonnegative integers, write $|\nu| = \sum \nu_i$ and let $x^\nu$ denote 
$x_1^{\nu_1}\cdots x_{2n}^{\nu_{2n}}$, a monomial of total degree $|\nu|$. 
For $f = \sum_\nu c_\nu x^\nu \in \wh A$ we define $\|f\| = \sup\{|c_\nu|\}$ (or $\infty$ if the coefficients are unbounded). Note that $\|f\|$ is finite for $f \in A$.

For a real number $M\ge 1$, let $\wh A_M \subset \wh A$ be the
subset consisting of those $f$ as above such that for all $\nu \in \N^{2n}$ we have $|c_\nu|
\leq M^{|\nu|}$.  Since the absolute value on $F_0$ is non-archimedean, $\wh A_M$ is a ring; and it is complete with
respect to the restriction of the $\mf m_0$-adic topology on $\wh A$.  Note also that $\wh A_M \subset \wh A_{M'}$ if $M < M'$.  In the case that $M = |t|^s$ for some (possibly negative) integer $s$, the subring $\wh A_M \subset \wh A$ is just the power series ring 
$\wh R_0[[x_1t^s, \ldots, x_{2n}t^s]]$.  In general, the next result shows that we can view the elements of $\wh A_M$ as power series functions that are defined and $t$-adically bounded by $1$ on the $t$-adic open disc of radius $M^{-1}$ about the origin in $\mbb A^{2n}(F_0)$.

\begin{lem} \label{convergent series}
{}
\renewcommand{\theenumi}{\alph{enumi}}
\begin{enumerate}
\item \label{rational geometric}
Let $f \in A_0 \subset \wh A$ satisfy $|f(0)| \le 1$.  Then for some $M \ge 1$ we have $f \in \wh A_M$ and $f=g/h$ for some $g,h \in A$ with $h \in \wh A_M^\times$. 
\item \label{convergent}
Let $M \ge 1$ and let $f = \sum_{\nu \in \N^{2n}} c_\nu x^\nu \in \wh A_M$.  If $a \in \A^{2n}(F_0)$ with $|a| < M^{-1}$ then the series 
\[f(a) := \sum_\nu c_\nu a^\nu\] 
converges $t$-adically to an element of $F_0$, of norm at most $1$. 
\item  \label{rational convergence}
In part~(\ref{rational geometric}), 
if $a \in \A^{2n}(F_0)$ with $|a| < M^{-1}$ then
the series $f(a)$ converges $t$-adically to the value $g(a)/h(a) \in F_0$.
\end{enumerate}
\end{lem}

\begin{proof}
(\ref{rational geometric})  
Since $A_0$ is the localization of $\wh R_0[x_1,\dots,x_n]$ at the ideal $(x_1,\dots,x_n)$, we may write $f=g/h$ with $g,h \in \wh R_0[x_1,\dots,x_n] \subset A$ and with $h \not\in \mf m_0$.  Here 
$\|g\| \le 1$ and $\|h\| = 1$, so $g,h \in \wh A_1$.
Since $A/\mf m_0$ is a field, there exists $h' \in A$ such that $hh'-1 \in \mf m_0$.  Writing $hh'=1-e$ with $e \in \mf
m_0 \subset A$, we see that the inverse to $h$ in $\wh A$ is given by $\sum_{i \geq 0} h'e^i$ (where this series converges in $\wh A$ because $e \in \mf m_0$).  So $f= \sum_{i \geq 0} gh'e^i \in \wh A$.
Let $M = \max\{1, \|h'\|, \|e\| \} < \infty$.   Thus $g, h, h', e \in \wh A_M$ (using that the nonconstant coefficients satisfy
the required condition by our choice of $M$, and the constant coefficients have absolute value at most $1$).
Furthermore, since $\wh A_M$ is a ring, each
term of $\sum h' e^i$ and of $\sum gh' e^i$ is also in $\wh A_M$.  Since $\wh A_M$ is complete with respect to the restriction of the $\mf m_0$-adic topology on $\wh A$, these series converge to 
elements of $\wh A_M \subset \wh A$.  Therefore, $f, h^{-1} \in \wh A_M$ and $h \in \wh A_M^\times$.

(\ref{convergent}) Since $f \in \wh A_M$, we have that $|c_\nu| \le M^{|\nu|}$ for each $\nu$.   Let $a \in \A^{2n}(F_0)$ with $m := |a| < M^{-1}$. 
Thus
$|c_\nu a^\nu| \le (mM)^{|\nu|} < 1$ for each $\nu$, since $0 \le mM < 1$.  Since $F_0$ is $t$-adically complete, the series $f(a)$ (which has finitely many terms of each total degree) converges to an element of $F_0$ of norm at most $1$.

(\ref{rational convergence}) Since $h \in A \cap \wh A_M^\times$, 
we have $h(a)h^{-1}(a)=1$ and so $h(a) \ne 0$.
Let $d > \deg(g)$ and let $C = \max\{\|g\|,\|h\|\}$.  Let $f_s$ be the polynomial truncation of the series $f \in \wh A$ modulo the terms of degree $\ge s$.  Thus the sequence $f_s(a)$ converges to some $c \in F_0$, by (\ref{convergent}).
If $s \ge d$, then $k_s := f_s h - g$ is a polynomial whose terms each have degree $\ge s$ and for which the coefficients of the terms of degree $j$ have absolute value at most $M^j C$.  With $m:=|a|$, the absolute values of the terms of degree $j$ in $k_s(a)$ are at most $(mM)^jC$, and so $|k_s(a)| \le (mM)^sC$.  Thus $k_s(a) \to 0$, since $0 < mM < 1$. That is, $f_s(a) h(a) \to g(a)$, or equivalently $c h(a) = g(a)$.  
Thus $c = g(a)/h(a)$, i.e.\ the series $f(a)$ converges to $g(a)/h(a)$.
\end{proof}

\begin{lem}\label{taylor works}
Suppose $f \in \wh A_M$ with $M \ge 1$, and write
\[f = c_{0,0} + L + \sum_{|\nu| \geq 2} c_\nu x^\nu\]
where $L$ is a linear form in $x_1,\dots,x_{2n}$ and all $c_\nu \in F_0$.
Let $s \ge 0$, let $0 < \varepsilon \leq |t|/M^2$, and suppose $a,  a' \in \A^{2n}(F_0)$ with $|a| \leq \varepsilon$ and $|a'| \leq \varepsilon|t|^s$. Then
\[\left|f(a + a') - f(a) - L(a')\right| \leq \varepsilon |t|^{s+1}.\]
\end{lem}

\begin{proof}
We may rearrange the quantity of interest as:
\[f(a + a') - f(a) - L(a') = \sum_{|\nu| \geq 2} c_\nu \left((a + a')^\nu - a^\nu \right).\]
Since the absolute value is non-archimedean, it suffices to show that for every
term $m = c_\nu x^\nu$ with $|\nu| \geq 2$ we have
\[|m(a + a') - m(a)| \leq \varepsilon|t|^{s+1}.\]

For a given $\nu$ with $|\nu| \ge 2$, consider the expression $(x + x')^\nu -
x^\nu$, regarded as a homogeneous element of degree $j = |\nu|$
in the polynomial ring
$F_0[x_1, \ldots, x_{2n}, x_1', \ldots, x_{2n}']$.
Since the terms of degree $j$ in $x_1, \ldots, x_{2n}$ cancel, the result is a sum of terms of the form $\lambda \ell$ where
$\lambda$ is an integer and $\ell$ is a monomial in the variables $x, x'$ with total degree $d$ in $x_1, \ldots, x_{2n}$ and total degree $d'$ in $x_1', \ldots, x_{2n}'$, such that $d
+ d' = j$ and $d < j$.  Hence $d' \ge 1$. Consequently, for each term of this form, 
\[|\lambda \ell(a, a')| \leq |\ell(a, a')| \leq \varepsilon^d(\varepsilon
|t|^s)^{d'} = \varepsilon ^ {j} |t|^{sd'}
\leq \varepsilon^{j}|t|^s.\]
Since $(a + a')^\nu - a^\nu$ is a sum of such terms, and the norm is
non-archimedean, we conclude $|(a + a')^\nu - a^\nu| \leq
\varepsilon^{j}|t|^s$.  

Since $m = c_\nu x^\nu$, it follows that  
\[|m(a + a') - m(a)| \leq |c_\nu|\varepsilon^{j} |t|^s  \leq
M^{j} \varepsilon^{j} |t|^s.\]
Now $\varepsilon \leq |t|/M^2$, so $\varepsilon^{j-1} \leq |t|^{j-1}/M^{2j - 2}$.
Since $|t| < 1$, $M \ge 1$, and $j \geq 2$, we have
\[\varepsilon^{j-1} \leq \frac{|t|^{j-1}}{M^{j + j - 2}} \leq \frac{|t|}{M^j}.\]
Rearranging this gives the inequality $(M\varepsilon)^j \leq \varepsilon |t|$ and so $(M \varepsilon)^j |t|^s \leq \varepsilon |t|^{s+1}$. Therefore
\[|m(a + a') - m(a)| \leq M^j \varepsilon^j |t|^s \leq \varepsilon |t|^{s+1},\]
as desired.
\end{proof}

For the remainder of this section, it will be convenient to write $y_i = x_{n+i}$ for $i=1,\dots,n$.  We will let $\nu = (\nu_1, \ldots, \nu_n)$ and $\rho = (\rho_1, \ldots, \rho_n)$ be $n$-tuples of non-negative integers; and for such $\nu, \rho$ we will write $|(\nu, \rho)| = \sum \nu_i + \sum
\rho_i$ and will let $x^\nu y^\rho$ denote 
$x_1^{\nu_1}\cdots x_n^{\nu_n} y_1^{\rho_1}\cdots y_n^{\rho_n}$, a monomial of total degree $|(\nu, \rho)|$.  An element of $\wh A$ will be written as $f = \sum_{\nu,
\rho} c_{\nu, \rho} x^\nu y^\rho$ with $c_{\nu, \rho} \in F_0$.

\begin{lem}\label{special taylor}
Let $f \in \mf m_0 A_0$, and suppose there is some $1 \leq i \leq n$ such that $f(a, 0) = a_i =
f(0, a)$ for all $a = (a_1, \ldots, a_n) \in F_0^n$ for which $f(a, 0)$ and $f(0,
a)$ converge. Then $f\in \wh A_M \subset \wh A$ for some $M \ge 1$, and its
expansion has the form
\[f = x_i + y_i + \sum_{|(\nu, \rho)| \geq 2} c_{\nu, \rho} x^\nu y^\rho.\] 
\end{lem}

\begin{proof}
By Lemma~\ref{convergent series}(\ref{rational geometric}), both~$f$ and
$g := f - x_i - y_i  \in \mf m_0 A_0$ lie in $\wh A_M$ for some $M \ge 1$; in particular, $g =
\sum_{\nu, \rho} c_{\nu, \rho} x^\nu y^\rho$
with each $|c_{\nu, \rho}| \le M^{|(\nu, \rho)|}$.  Here $g$
converges in a $t$-adic neighborhood of $(0,0)$, on which $g(a,b) = 0$ if $a = 0$ or $b = 0$.  To prove the result it 
suffices to show that $c_{\nu, \rho}=0$ for $|(\nu, \rho)| < 2$.  This is automatic for $c_{0,0}$ since $g \in \mf m_0 A_0$.  It remains to
show that $c_{\nu, \rho} = 0$ for $|(\nu, \rho)| =1$.

We argue by contradiction. Suppose that there exists $(\nu_0, \rho_0)$ such that
$c_{\nu_0, \rho_0} \neq 0$ with $|(\nu_0, \rho_0)| = 1$. Without loss of generality, we
may assume that $\nu_0 = (1,0,\dots,0)$ and $\rho_0 = (0,0,\dots,0)$.  Choose $0 < m < 1$ such that $m \leq |c_{\nu_0,
\rho_0}|$, and $N > 0$ such that $|t^N| < m/M^2$.  Let $v= 
(t^N,0, \ldots, 0) \in \mbb
A^{2n}$;  thus $g(v) = 0$. 
Also, $|L(v)| \geq m|t^N| > 0$, where $L$ is the sum of the terms of $g$ of degree $1$.  So $L(v) \ne 0$.

Now let $h = c_{\nu, \rho} x^{\nu} y^{\rho}$ be an arbitrary
term of $g$ whose degree $j:=|(\nu, \rho)|$ is at least $2$. We claim $|h(v)| <
|L(v)|$.  Showing this for all such $h$ would imply that $|g(v)| = |L(v)|$.
Since $g(v) = 0 \ne L(v)$, this
would lead to a contradiction.  

To verify the claim, we may assume
$h(v) \neq 0$.  Using the definition of $v$, we see directly that $h = c x_1^j$ for some $c \in F_0$, and that $|h(v)| = |c||t^N|^j$.  Here $|c| \le M^j$ since $g \in \wh A_M$ and $h$ is a term of $g$.  We compute
\[\frac{|L(v)|}{|h(v)|} \geq \frac{m|t^N|}{|c||t^N|^j} \geq
\frac{m}{M^j}\frac{1}{|t^N|^{j-1}}.\]
Since $|t^N| < m/M^2$, we have $1/|t^N| > M^2/m$ and so 
$1/|t^N|^{j-1} >
M^{2(j-1)}/m^{j-1}$. Combining this with the above, 
\[\frac{|L(v)|}{|h(v)|} \geq \frac{m}{M^j}\frac{1}{|t^N|^{j-1}} >
\frac{m}{M^j}\frac{M^{2j - 2}}{m^{j-1}} = \frac{M^{j-2}}{m^{j-2}} \ge 1,\]
because $j \ge 2$, $M \ge 1$, and $0 < m < 1$.  So $|L(v)| >
|h(v)|$ as desired.
\end{proof}

For the next result and for use in the next section, we make the following hypothesis, continuing under the notation introduced at the beginning of the current section:

\begin{hyp}\label{factorization hyp}
We assume that the complete discrete valuation ring $\wh R_0$ contains a
subring $T$ which is also a complete discrete valuation ring having uniformizer $t$, and that 
$F_1, F_2 $ are subfields of $F_0$ containing $T$.  We further assume that $\m V \subset F_1 \cap \wh R_0$, $\m W \subset F_2
\cap \wh R_0$ are $t$-adically complete $T$-submodules satisfying
$\m V + \m W  = \wh R_0$. 
\end{hyp}

The main theorem of this section is the following decomposition result, which is related to \cite{HH:FP}, Proposition~3.2 (with $\mbb A_{F_0}^n$ here corresponding to the affine space of square matrices of a given size):

\begin{thm} \label{abstract small factorization}
Under Hypothesis~\ref{factorization hyp}, let $f : \mbb A^n_{F_0} \times \mbb A^n_{F_0} \dra \mbb A^n_{F_0}$ be an $F_0$-rational map that
is defined on a Zariski open set $U \subseteq \mbb A^n_{F_0}
\times \mbb A^n_{F_0}$ containing the origin $(0,0)$. Suppose further that 
$f(u, 0) = u = f(0, u)$ whenever $(u, 0)$ (resp.\ $(0,u)$) is in $U$.  Then there is a real number $\varepsilon > 0$ 
such that for all $a \in \mbb A^n(F_0)$ with $|a| \leq \varepsilon$, there exist $v \in \m V^n$ and $w \in \m W^n$ such that $(v,w) \in U(F_0)$ and $f(v, w) = a$.
\end{thm}

\begin{proof}
Write the function $f$ as an $n$-tuple $(f_1, \ldots, f_n)$ with $f_i \in A_0$.  In fact $f_i \in \mf m_0 A_0$ since $f_i(0,0)=0$.
So by Lemma~\ref{convergent series}(\ref{rational geometric}), there is a real number $M \ge 1$ such that
$f_i \in \wh A_M$ for all $i$; and by Lemma~\ref{special taylor}, 
\[f_i = x_i + y_i + \sum_{|(\nu, \rho)| \geq 2} c_{\nu, \rho, i} x^\nu
y^\rho \eqno{(*)}\]
for some $c_{\nu, \rho, i}$ in $F_0$. 

As noted at the beginning of this section, the $t$-adic topology on affine space is finer than the Zariski topology.  So there exists
$\delta > 0$ such that 
$(v,w) \in U(F_0)$ for all $v,w \in \mbb A^n(F_0)$ satisfying
$|(v,w)| \le \delta$.  
Choose $N > 0$ such that $|t^N| \leq \min\{|t|/M^2,
\delta\}$, and set $\varepsilon = |t^N|$.  In particular, whenever $|(v,w)|
\le \varepsilon$, the point $(v,w)$ lies in $U(F_0)$, and hence $f(v,w)$ is defined.

Now suppose $a = (a_1, \ldots, a_n) \in \mbb A^n(F_0)$ with $|a| \leq
\varepsilon$. We will inductively construct sequences of elements 
$v_j = (v_{1,j}, \ldots, v_{n, j}) \in \m V^n$, 
$w_j = (w_{1, j}, \ldots, w_{n, j}) \in \m W^n$, with
$j \geq 0$, such that $v_0=w_0=(0,\dots,0)$ and
\begin{enumerate}
\item $|(v_j,w_j)| \le \varepsilon$ for all $j \ge 0$;
\item $|v_j - v_{j-1}|, |w_j - w_{j-1}| \leq \varepsilon |t|^{j-1}$ for all $j \ge 1$; and
\item $|f_i(v_j, w_j) - a_i| \leq \varepsilon|t|^j$ for all $j \ge 0$.
\end{enumerate}
Since the $T$-modules $\m V^n, \m W^n$ are $t$-adically complete, the second condition ensures that $v = \lim\limits_{j \to \infty} v_j$ and 
$w = \lim\limits_{j \to \infty} w_j$ exist in $\m V^n$ and $\m W^n$. 
The first condition shows moreover that $|(v,w)| \le \varepsilon$, so that $(v,w) \in U(F_0)$ and $f(v,w)$ is defined.  Finally,
the third condition implies that $f(v, w) = a$.  Thus it suffices to construct such sequences.

It follows from Lemma~\ref{convergent series}(\ref{convergent}) that since at each stage we will have
$|(v_j, w_j)| \le \varepsilon \leq |t|/M^2 < 1/M$, the power series expressions for
$f(v_j, w_j)$ are convergent.  By Lemma~\ref{convergent series}(\ref{rational convergence})
we may identify the limits of these
evaluated power series with the values of the original rational functions.

Observe that the first and third conditions hold for $j=0$.  Now assume inductively that for some $j \ge 0$ we have chosen $v_j, w_j$ satisfying the three asserted 
conditions (except the second, if $j=0$).  Define $b_j = (b_{1, j}, \ldots, b_{n, j}) = a - f(v_j, w_j)$. By the third condition on $(v_j,w_j)$, we have $|b_j| \leq \varepsilon|t|^j = |t|^{N+j}$. Write $b_j = t^{N + j} u_j$ with $u_j \in \wh R_0^n$. By Hypothesis~\ref{factorization hyp}, we may write $u_j = v_{j+1}'
+ w_{j+1}'$ for $v_{j+1}' \in \m V^n$, $w_{j+1}' \in \m W^n$. 

Let $v_{j+1} = v_j + t^{N + j} v_{j+1}'$ and $w_{j+1} = w_j + t^{N + j}
w_{j+1}'$. It is immediate by construction that $|v_{j+1} - v_j|, |w_{j+1} - w_j| \leq \varepsilon |t|^j$ since $|t|^N = \varepsilon$. 
This proves the second condition on $(v_{j+1},w_{j+1})$.
Since $|v_j|, |w_j| \leq
\varepsilon \leq |t|/M^2$ (by the first condition on $(v_j,w_j)$) and since $|t^{N + j} v_{j+1}'|, |t^{N + j} w_{j+1}'| \leq
\varepsilon |t|^{j}$, it follows by equation~${(*)}$ and Lemma~\ref{taylor works} that
\begin{equation*}
\begin{split}
|f_i(v_{j+1}, w_{j+1}) - a_i| 
&= |f_i(v_{j+1}, w_{j+1}) - f_i(v_j, w_j) - b_{i, j}| \\
= \left|f_i(v_j + t^{N + j} v_{j+1}'\right., 
& \left. w_j + t^{N + j} w_{j+1}') - f_i(v_j, w_j)
- t^{N + j}(v_{j+1}' + w_{j+1}')\right| \\
&\leq \varepsilon |t|^{j+1},
\end{split}
\end{equation*}
proving the third condition on $(v_{j+1},w_{j+1})$.  The first condition on $(v_{j+1},w_{j+1})$ holds by the second condition  on $(v_{j+1},w_{j+1})$ together with the first condition on $(v_j,w_j)$, since the norm is non-archimedean and $j \ge 0$.
\end{proof}

The above decomposition theorem will be used in the next section to extend \cite{HH:FP},
Proposition~3.2, which used an additive decomposition to
provide a factorization of matrices in $\GL_n$.  In applying
Theorem~\ref{abstract small factorization} above to obtain factorization in more general rational linear algebraic groups $G$ (Theorem~\ref{factorization} below), we will identify $G$ birationally with an open subset of some affine space, with $f$ above being the map there that corresponds to multiplication in $G$.

%%%%%%%%%%%%%%%%%%%%%%%%%%%%%%%%%%%%%%%%%%%%%%%%%%%%%%%%%%%%%%%%%%%%%%%%%%%%%%%%
%Section 3
\section{Factorization and a local-global principle} \label{rational factorization}
%%%%%%%%%%%%%%%%%%%%%%%%%%%%%%%%%%%%%%%%%%%%%%%%%%%%%%%%%%%%%%%%%%%%%%%%%%%%%%%%

We say that a connected linear algebraic group defined over a field $F$ is \textit{rational}
if it is rational as an $F$-variety. 
In this section we prove factorization theorems for such groups (Theorems~\ref{smooth factorization} 
and~\ref{general factorization}), generalizing results of \cite{HH:FP} about the rational group $\GL_n$.  
The key step is Theorem~\ref{factorization}, which relies on Theorem~\ref{abstract small factorization}.  
Afterwards, in Theorem~\ref{torsor_injective}, we apply this factorization to obtain a local-global principle 
for homogeneous spaces for rational groups.

\begin{lem}\label{zariski rational density}
Let $G$ be a rational connected linear algebraic group over an infinite field $F$, 
let $F_0$ be an extension field of $F$, and let $g \in G(F_0)$. Then there exists
a Zariski open subset $Y \subseteq G$ such that $g \in Y(F_0)$ and such that
$Y$ is $F$-isomorphic to an open subset of affine space over $F$. 
\end{lem}

\begin{proof}
Since $G$ is rational, there exists a non-empty irreducible Zariski open subset $Y' \subseteq G$ that is
isomorphic to an open subset of affine space.  Since $F$ is infinite, every non-empty
open subset of affine $F$-space contains an $F$-point. Consequently, there exists a point $y \in Y'(F)$,
and the $F_0$-scheme $y^{-1}g^{-1}Y_{F_0}' \cap Y'_{F_0}$ is a Zariski dense
open subset of $Y'_{F_0}$. 
Since $Y'$ is
$F$-isomorphic to an open subset of affine space, and since $F$ is infinite, it follows that $Y'(F)$ is
dense in $Y'(F_0)$ with respect to the Zariski topology. Therefore there exists
$y' \in Y'(F)$ such that $y' \in y^{-1}g^{-1}Y'(F_0)$. That is to say, 
$g \in Y'(F_0)(yy')^{-1}$.  Setting $Y = Y'(yy')^{-1}$, we find that $Y$ is an
$F$-subscheme of $G$ such that $g \in Y(F_0)$ and $Y \cong Y'$ is $F$-isomorphic to
an open subvariety of affine space.
\end{proof}

The following factorization theorem, which extends \cite{HH:FP}, Proposition~3.2, 
to more general rational linear algebraic groups than $\GL_n$, relies on Theorem~\ref{abstract small factorization} above:

\begin{thm} \label{factorization}
Under Hypothesis~\ref{factorization hyp}, assume that $F$ is a subfield of $F_1 \cap F_2$ that contains $T$, and that $F_1$ is
$t$-adically dense in $F_0$.
Let $G$ be a rational connected linear algebraic group defined over $F$. Then for any $g \in
G(F_0)$ there exist $g_i \in G(F_i)$, $i = 1, 2$, such that $g_1 g_2 = g$.
\end{thm}

\begin{proof}
Since $G$ is rational, there is a Zariski dense open subset $U'$ of $G$ that is $F$-isomorphic to a Zariski open subset $U$ of 
$\mbb A_F^n$, where $n$ is the dimension of $G$.  After translating, we may assume that $U'$ contains the identity 
$e \in G$ and that the $F$-isomorphism $\phi: U' \to U \subseteq \mbb A_F^n$ takes $e$ to the
origin in $\mbb A^n(F)$.
Consider the group multiplication map $\mu : G
\times G \to G$, and let $\til U' = \mu^{-1}(U') \cap
(U' \times U') \subseteq G \times G$.  Note that $\til U'$ is a Zariski open subset of $G \times G$ that contains the point 
$(e,e)$, and that
$\mu(\til U') \subseteq U'$. 
The isomorphism $\phi: U' \to U$ induces an isomorphism
$\phi \times \phi|_{\til U'}: \til U' \to \til U$ for some dense open subset 
$\til U \subseteq \mbb A^n_F \times \mbb A^n_F = \mbb A^{2n}_F$.  Hence there exists a morphism $f:\til U \to U$ such that the following diagram commutes:
\[\xymatrix{
G \times G & \supseteq &
\til U' \ar[rr]^{\mu|_{\til U'}} \ar[d]_{\phi \times \phi|_{\til U'}} & & U' \ar[d]^{\phi} & \subseteq & G
\\ 
\mbb A^{2n}_F & \supseteq &
\til U \ar[rr]_{f} & & U & \subseteq & \mbb A^n_F}\]
Since $\mu(g, e) = g = \mu(e, g)$ for $g \in G$, it follows that $f(v, 0) = v = f(0, v)$ when $(v,0)$ and
$(0,v)$ are in $\til U$.  Consequently, it follows from
Theorem~\ref{abstract small factorization} (with $\til U$ here playing the role of $U$ there) 
and the assumptions of Hypothesis~\ref{factorization hyp}
that there is an $\varepsilon > 0$
such that for $g \in U'(F_0)$ with
$|\phi(g)| \le \varepsilon$, there exist $v \in \m V^n$, $w \in \m W^n$ with
$(v, w) \in \til U(F_0)$ such that $f(v, w) = \phi(g)$.  Thus $v \in U(F_1)$ and $w \in U(F_2)$.

To prove the theorem, consider first the special case that $g \in G(F_0)$ satisfies $|\phi(g)| \le \varepsilon$.  
If we set $g_1 = \phi^{-1}(v)$ and $g_2 =
\phi^{-1}(w)$ for $v, w$ as above, then $g_i \in G(F_i)$ for $i = 1, 2$, and $(g_1,g_2) \in \til U'$. But now
we have $g_1 g_2 = g\in G(F_0)$ by the above commutative diagram, as desired.

The general case reduces to the above special case by a classical argument (e.g.\ see \cite{Kneser}).  Namely, by Lemma~\ref{zariski
rational density}, there is an open $F$-subset $Y \subseteq G$ such that $g\in Y(F_0)$, together with an open immersion 
$\psi : Y \to \mbb A^n_F$.  Now $\A^n(F_1)$ is $t$-adically dense in $\A^n(F_0)$, and the $t$-adic topology is finer than the 
Zariski topology.  So since
$\phi(e) = 0 \in \A^n_F$, there exists $h \in Y(F_1) \subseteq G(F_1)$ such that $|\phi(h^{-1} g)| \le \varepsilon$.  
By the first part, $h^{-1} g = g_1' g_2'$ with $g_i' \in
G(F_i)$, $i = 1, 2$.  Setting $g_1 = hg_1' \in G(F_1)$ and $g_2 =
g_2' \in G(F_2)$ gives the desired conclusion $g = g_1 g_2 \in G(F_0)$.
\end{proof}

In order to apply this result, we recall the following notation and terminology, which arose in the context of patching in \cite{HH:FP}:

\begin{notn} [\cite {HH:FP} \rm{(Sections~4 and 6)}] \label{notation}
Let $T$ be a complete discrete valuation ring with uniformizer $t$ and residue field $k$, and let 
$\wh X$ be a normal irreducible projective $T$-curve with function field $F$ and with closed fiber $X$.  
Given an irreducible component $X_0$ of $X$ with generic point $\eta$, consider the local ring of $\wh X$ at $\eta$.  
For a (possibly empty) proper subset $U$ of $X_0$, we let $R_U$ denote the subring of this local ring consisting of the 
rational functions that are regular at each point of $U$. In particular,
$R_\varnothing$ is the local ring of $\wh X$ at the generic point of the
component $X_0$.  The $t$-adic completion of $R_U$ is denoted by $\wh R_U$.
If $P$ is a closed point of $X$, we write $R_P$ for the
local ring of $\wh X$ at $P$, and $\wh R_P$ for its completion at its maximal ideal.  
(Note the distinction $\wh R_P$ vs.\ $\wh R_{\{P\}}$.)  A
height~$1$ prime $\wp$ of $\wh R_P$ that contains $t$ determines a {\it
branch} of $X$ at $P$, i.e.\ an irreducible component of the pullback of $X$
to $\Spec\,\wh R_P$.  Similarly the contraction of $\wp$ to the local ring of
$\wh X$ at $P$ determines an irreducible component $X_0$ of $X$, and we say
that $\wp$ {\it lies on} $X_0$.  Note that a branch $\wp$ uniquely determines a closed point $P$ and an irreducible component $X_0$.  
In general, there can be several branches $\wp$ on $X_0$ at a point $P$; but if $X_0$ is smooth at $P$ then there is
a unique branch $\wp$  on $X_0$ at $P$.  We write $\wh R_\wp$ for
the completion of the localization of $\wh R_P$ at $\wp$; thus $\wh R_P$ is contained in $\wh R_\wp$, 
which is a complete discrete valuation ring.  

Since $\wh X$ is normal, the local ring $R_P$ is integrally closed and hence unibranched; and since $T$ is a complete 
discrete valuation ring, $R_P$ is excellent and hence $\wh R_P$ is a domain (\cite{EGAIV2}, Scholie~7.8.3(ii, iii, vii)).   
For nonempty~$U$ as above and $Q \in U$, $\wh R_U/t^n\wh R_U \to \wh R_Q/t^n\wh R_Q$ is injective for all $n$ and hence 
$\wh R_U \to \wh R_Q$ is also injective. Thus $\wh R_U$ is also a domain. Note
that the same is true if $U$ is empty. The fraction fields
of the domains $\wh R_U$, $\wh R_P$, and $\wh R_\wp$ will be denoted by $F_U$,
$F_P$, and $F_\wp$.  

If $\wp$ is a branch at $P$ lying on the closure of $U \subset X_0$, then there are natural inclusions of $\wh R_P$ and $\wh R_U$ into $\wh R_\wp$, and hence of $F_P$ and $F_U$ into $F_\wp$.  The inclusion of $\wh R_P$ was observed above; for $\wh R_U$, note that the localizations of $R_U$ and of $R_P$ at the generic point of $X_0$ are the same (viz.\ $R_\varnothing$); and this localization is naturally contained in the $t$-adically complete ring $\hat R_\wp$.  Thus so is 
$R_U$ and hence its $t$-adic completion $\wh R_U$.   
\end{notn}

\begin{thm} [(Factorization over smooth curves)] \label{smooth factorization}
Let $T$ be a complete discrete valuation ring, let $\wh X$ be a smooth connected projective $T$-curve with function field $F$ 
and closed  fiber $X$.  Define fields $F_i$, $i=0,1,2$, by one of the following:
\begin{enumerate}
\item \label{global}
$F_i =F_{U_i}$ where $U_1, U_2$ are proper subsets of $X$, $U_1 \cup U_2 = X$, and $U_0 = U_1 \cap U_2$.  (Note that here $F_0$  
is not a complete discretely valued field unless $U_0 = \varnothing$.)
\item \label{local-global}
$F_1 = F_P$, $F_2 = F_U$, and $F_0 = F_\wp$, where $P$ is a closed point of $X$
  with complement $U \subset X$, and $\wp$ corresponds to the (unique) branch of $X$ at $P$. 
\item \label{local}
$F_1$ and $F_0$ are as in (\ref{local-global}) and $F_2 = F_\varnothing$. 
\end{enumerate}
Let $G$ be a rational connected linear algebraic group defined over $F$.  Then for any $g \in
G(F_0)$ there exist $g_i \in G(F_i)$, $i = 1, 2$, such that $g_1 g_2 = g$.
\end{thm}

\begin{proof}
It suffices to show that the hypotheses of Theorem~\ref{factorization} hold in each of the three parts of the above assertion.  

For part~(\ref{global}), write $\wh R_i = \wh R_{U_i}$.  We first assume that $U_0$ is empty, so that $\wh R_0$ is a complete discrete valuation ring with uniformizer $t$.  In this situation let
$P$ be a closed point of $U_1$ and let $\wh P$ be a lift of $P$ to $\wh X$ (i.e.\ an effective prime divisor on $\wh X$ whose restriction to $X$ is $P$; see \cite{HH:FP}, Section~4.1).  Let $g_X$ be the genus of $X$, pick a non-negative integer $N > 2g_X-2$, and let $\m V = L(\Spec \wh R_1, N\wh P)$, the $T$-submodule of $F_1$ consisting of rational functions on $\Spec \wh R_1$ whose pole divisor is at most $N \wh P$.  Note that $\m V \subset \wh R_0$ since these rational functions do not have poles along the closed fiber $(t)$.  Let $\m W = \wh R_2$.  Then Hypothesis~\ref{factorization hyp} holds for these rings and modules, by \cite{HH:FP}, Proposition~4.5.  

To complete the proof of part~(\ref{global}) in this case, it suffices to verify that the hypotheses of Theorem~\ref{factorization} are satisfied, i.e.\ that $F_1$ is $t$-adically dense in $F_0$.
Since the fraction field of $\wh R_1/t\wh R_1$ is the same as   
$\wh R_0/t\wh R_0$ (viz.\ the function field of $X$), it follows from
\cite{HH:FP}, Lemma~3.1(a), that the ring $R_0 := \wh R_0 \cap F_1$ is $t$-adically dense in $\wh R_0$.  For the density of $F_1$ in $F_0$, let $x \in F_0$; since $\wh R_0$ is a discrete valuation ring with uniformizer $t$, we may write $x = t^{-n} y$ for some $n \in \mbb Z$ and $y \in \wh R_0$.  By the density of $R_0$ in $\wh R_0$, for any $\ell > 0$ there exists $x_0 \in R_0$ such that $x_0 - y \in t^{\ell+n} \wh R_0$. It then follows that $t^{-n}x_0 - x \in t^\ell \wh R_0$. But $t^{-n}x_0 \in F_1$, since the field $F_1$ contains $R_0$.  So $F_1$ is indeed dense in $F_0$, finishing the proof of part~(\ref{global}) of the theorem in this special case.

More generally, if $U_0$ is not necessarily empty, then we proceed as follows (paralleling the proof of \cite{HH:FP}, Theorem~4.10).  Let $U_2' =  U_2 \smallsetminus U_0$, and write $F_2' = F_{U_2'}$ and $F_0' = F_\varnothing$.  Any $g \in G(F_0)$ lies in $G(F_0')$, and so by the above special case we may write $g = g_1g_2$ with $g_1 \in G(F_1) \leq G(F_0)$ and $g_2 \in G(F_2')$.  But also $g_2 = g_1^{-1}g \in G(F_0)$; and $F_2' \cap F_0 = F_2$ by Theorem~4.9 of \cite{HH:FP} since $U_2' \cup U_0 = U_2$.  So actually $g_2 \in G(F_2)$, finishing the proof of part~(\ref{global}).

For part (\ref{local}), take $\m V = \wh R_P$ and $\m W = \wh R_\varnothing$.   Hypothesis~\ref{factorization hyp} holds by \cite{HH:FP}, Lemma~5.3.  Also, the fraction field of $\wh R_P/t\wh R_P$ is $\wh R_\wp/t\wh R_\wp$, by \cite{HH:FP}, Lemma~5.2(d).  (Those results apply since if we let $\wh R = \wh R_{\{P\}}$, then it is straightforward to check that the rings $\wh R_1$, $\wh R_2$, $\wh R_0$ obtained from $\wh R$ in \cite{HH:FP}, Notation~5.1, are the same as the rings $\wh R_P$, $\wh R_\varnothing$, $\wh R_\wp$ here.)
So as in the proof of (\ref{global}), $F_1$ is $t$-adically dense in $F_0$.  Thus Theorem~\ref{factorization} implies the assertion.

Part (\ref{local-global}) is now immediate from the other two parts.
Specifically, by (\ref{local}) we may factor any element $g \in G(F_0)$ as
$g_1'g_2'$ with $g_1' \in G(F_P)$ and $g_2' \in G(F_\varnothing)$.  Taking $U_1
= \{P\}$ and $U_2 = U = X \smallsetminus U_1$, by (\ref{global}) we may then
factor $g_2'$ as $g_1''g_2$ with $g_1'' \in G(F_P)$ (since
$F_{U_1}=F_{\{P\}}\subset F_P$) and $g_2 \in G(F_U)$.  
Writing $g_1=g_1'g_1'' \in G(F_P)$ gives the desired factorization $g=g_1g_2$.
\end{proof}

The above factorization theorem generalizes results of \cite{HH:FP} about $\GL_n$ to rational connected linear algebraic groups $G$. 
Parts~(\ref{global}) and~(\ref{local}) for $\GL_n$ were respectively shown in
Theorem~4.10 and Theorem~5.4 of \cite{HH:FP} (which in the latter case again used the above comment about the rings in \cite{HH:FP}, Notation~5.1).  Also, if $\wh X = \PP^1_T$, we can just take $\m V = \wh R_1$ in case~(\ref{global}) of the above proof, corresponding to choosing $N=0$.

As in Section~6 of \cite{HH:FP}, the second part of the above result will next be extended to curves $\wh X$ that are not necessarily smooth, and to the case where several points are chosen.  To do this, we will choose a finite morphism $\wh X \to \PP^1_T$, so that the function field $F$ of $\wh X$ is a finite extension of the function field $F'$ of $\PP^1_T$.  We will then relate linear algebraic groups over $F$ to linear algebraic groups over $F'$, using the restriction of scalars functor $R_{F/F'}$ that takes affine varieties over $F$ to affine varieties over $F'$, and which is characterized by the functorial isomorphism $\alpha_{X,Z}:\Hom_{F'}(Z,R_{F/F'}(X))   
\to \Hom_F(Z \times_{F'} F,X)$ (see \cite{BLR:Neron}, Section~7.6, Theorem~4,
which does not require separability of $F/F'$).  It will be convenient to use the following notation:

\begin{notn}  \label{P1 notation}
In the context of Notation~\ref{notation}, assume that $f : \wh X \to \mbb P^1_T$ is a finite morphism such that 
$\mc P := f^{-1}(\infty)$ contains all points at which distinct irreducible components of the closed fiber $X \subset \wh X$ meet.  
(Such an $f$ exists by 
\cite{HH:FP}, Proposition 6.6.)  We let $\mc U$ be the collection of
irreducible components $U$ of $f^{-1}(\mbb A_k^1)$, and let $\mc B$ be the collection of all branches $\wp$ at the points of $\mc P$.
\end{notn}

\begin{thm} [(Simultaneous factorization for curves)] \label{general factorization}
Let $\wh X$ be a normal connected projective $T$-curve
and let $f : \wh X \to \mbb P^1_T$ a finite morphism, in the context of Notation~\ref{P1 notation}.
Let $G$ be a rational connected linear algebraic group
over  the function field $F$ of $\wh X$, and suppose that for every branch $\wp \in \mc B$
we are given an element $g_\wp \in G(F_\wp)$. Then we may find an element $g_P \in G(F_P)$ for each $P \in \mc P$, 
and an element $g_U \in G(F_U)$ for each $U \in \mc U$, 
such that for every branch $\wp \in \mc B$ at a point $P \in \mc P$ with $\wp$
lying on the closure of some
$U \in \mc U$, we have $g_\wp = g_P g_U$ with respect
to the natural inclusions $F_P, F_U \to F_\wp$.
\end{thm}

To avoid possible confusion, we emphasize that each $g_P$ (resp.~$g_U$) depends only on $P$ (resp.~$U$); but that the identity $g_\wp = g_P g_U$ takes place in $G(F_\wp)$, where we view $g_P$ and $g_U$ as elements of $G(F_\wp)$ via the respective inclusions of $G(F_P)$ and $G(F_U)$ that are induced by the corresponding inclusions of fields.  Thus if $\wp, \wp'$ are each branches at $P$ lying on the closure of $U$ (e.g.\ if $P$ is a nodal point on an irreducible component of $X$), then the products 
$g_\wp = g_P g_U$ and $g_{\wp'} = g_P g_U$ take place over different fields $F_\wp$, $F_{\wp'}$, with respect to different inclusions.

\begin{proof}
Let $F'$ be the function field of $\mbb P^1_T$. Thus $F$, the function field of $\wh X$, is a finite field extension of $F'$ via $f$.  Under Notation~\ref{notation} for $\PP^1_T$, we may consider the rings $\wh R_{\infty}$, $\wh R_{\mbb A_1}$, and $R_{\wp'}$ where
$\wp'$ is the branch at $\infty$ defined by the closed fiber.  
Let $F_1'$, $F_2'$, and $F_0'$ be the corresponding fraction fields.

Let $G' := R_{F/F'}(G)$.  By functoriality of $R_{F/F'}$, the $F'$-variety $G'$ is a linear algebraic group (e.g.\ see \cite{Milne}, Section~1) and it is rational.  By the defining property of $R_{F/F'}$, there is a natural isomorphism $G'(F_0') = G(F_0' \otimes_{F'} F)$.  Since $F_0' \otimes_{F'} F = \prod_{\wp} F_{\wp}$ by \cite{HH:FP}, Lemma~6.2(a), we find:
\[G'(F_0') = G(F_0' \otimes_{F'} F) = G(\prod_{\wp} F_{\wp}) = \prod_{\wp}
G(F_\wp).\]
Similarly,
\[G'(F_1')  = \prod_P G(F_P), \ \ \ \ \ G'(F_2) = \prod_U G(F_U),\]
via $F_1' \otimes_{F'} F = \prod_P F_P$ and 
$F_2' \otimes_{F'} F = \prod_U F_U$ (\cite{HH:FP}, Lemma~6.2(a)).
In particular, we may identify our tuple $(g_\wp) \in \prod_{\wp} G(F_\wp)$
with an element $g_0 \in G'(F_0')$. By Theorem~\ref{smooth factorization}(\ref{local-global}), there exist $g_1 \in G'(F_1')$ and
$g_2 \in G'(F_2')$ such that $g_0 = g_1 g_2$.  Again using the above
identifications, the element $g_1$ corresponds to a tuple $(g_P) \in \prod_{P
\in \mc P} G(F_P)$ and $g_2$ corresponds to a tuple $(g_U) \in \prod_{U \in
\mc U} G(F_U)$.  By \cite{HH:FP}, Lemma~6.2(b), the above isomorphisms on $F_i' \otimes_{F'} F$ (for $i=0,1,2$) are compatible, with respect to the inclusions of $\prod F_P$ and $\prod F_U$ into $\prod F_\wp$, and of $F_1', F_2'$ into $F_0'$.  So by the functoriality of $R_{F/F'}$, the above factorization $g_0 = g_1 g_2 \in G'(F_0')$ yields the desired equality $g_\wp = g_P g_U \in G(F_\wp)$ for each point $P \in \mc P$, each component $U \in \mc U$, and each branch
$\wp$ at $P$ lying on the closure of $U$.
\end{proof}

We continue to adopt Notations~\ref{notation} and~\ref{P1 notation}, concerning  a normal projective $T$-curve $\wh X$ 
with function field $F$ and associated sets $\mc P, \mc U, \mc B$.

In what follows, if a linear algebraic group $G$ acts on a variety $H$ over a field $F$, we will say that $G$ \textit{acts transitively on the points of} $H$ if for every field extension $E$ of $F$ the induced action of the group $G(E)$ on the set $H(E)$ is transitive.  (See also Remark~\ref{homog rk}.)

\begin{thm}[(Local-global principle for homogeneous spaces)]\label{torsor_injective}
Let $G$ be a rational connected linear algebraic group over $F$ 
that acts transitively on the points of an $F$-variety $H$. Then in the context of Notation~\ref{P1 notation},
$H(F) \neq \varnothing$ if and only if $H(F_P) \neq \varnothing$ for each
$P \in \mc P$ and $H(F_U) \neq \varnothing$ for each $U \in \mc U$.
\end{thm}

\begin{proof}
If $H(F)$ is non-empty, then so are each $H(F_P)$ and $H(F_U)$, since $F$ is contained in $F_P$ and $F_U$. 

For the converse, pick a point $h_P \in H(F_P)$ for each $P \in \mc P$
and a point $h_U \in H(F_U)$ for each $U \in \mc U$.  
For each $\wp \in \mc B$, as observed in Notation~\ref{notation} there is a unique point $P \in \mc P$ and a unique irreducible component $X_0$ of $X$ such that $\wp$ is a branch at $P$ that lies on $X_0$.  The component $X_0$ is the closure of a unique $U \in \mc U$, which is thus also determined by $\wp$.  
Here we can view $h_P$ and $h_U$ as points of $H(F_\wp)$ via the inclusions of 
$F_P$ and $F_U$ into $F_\wp$.  Since $G$ acts transitively on the points of $H$, there is an element $g_\wp \in G(F_\wp)$ such that $g_\wp(h_U) = h_P$ in $H(F_\wp)$.   Since $G$ is rational and we are in the situation of Notation~\ref{P1 notation},
Theorem~\ref{general factorization} implies that there is a collection of group elements $g_P \in G(F_P)$ for all $P \in \mc P$ and
$g_U \in G(F_U)$ for all $U \in \mc U$, 
such that for every branch $\wp$ at $P$ on the closure of $U$ we have $g_\wp = g_Pg_U$.  
Let $h_P' = g_P^{-1}(h_P) \in H(F_P)$ and $h_U' = g_U(h_U) \in H(F_U)$.  Thus if $P, U, \wp$ are a triple as above, then $h_P'$ and $h_U'$ become identified with the same element $h_\wp' \in H(F_\wp)$ under the inclusions of $H(F_P)$ and $H(F_U)$ into $H(F_\wp)$. 
(Here $h_\wp'$ depends only on $\wp$ since $\wp$ determines $P$ and $U$.)  

We claim that there is an affine Zariski open subset $\Spec A \subseteq H$ that contains the points $h_P'$, $h_U', h_\wp'$ for all $P,U,\wp$.  This is clear if $H$ is quasi-projective, since this set of points is finite.  For a more general variety $H$, observe that if $\xi_1, \xi_2 \in \mc P \cup \mc U \cup \mc B$ are related by being members of a common triple $P, U, \wp$ as above, then any affine open subset of $H$ that contains $h_{\xi_1}'$ must also contain $h_{\xi_2}'$ (since they define the same point in $H(F_\wp)$).  But since the closed fiber of the curve $\wh X$ is connected, {\it any} two elements $\xi_1, \xi_2 \in \mc P \cup \mc U \cup \mc B$ are in the transitive closure of this relation.  This proves the claim.

Let $\phi_P : A \to F_P$, $\phi_U : A \to F_U$, and $\phi_\wp: A \to F_\wp$ be the homomorphisms corresponding to the points $h_P' \in H(F_P)$, $h_U' \in H(F_U)$, and $h_\wp' \in H(F_\wp)$.  
Thus if $\wp$ is a branch at $P$ on the closure of $U$, the maps $\phi_P$ and $\phi_U$ each induce the homomorphism $\phi_\wp:A \to F_\wp$ via the inclusions $F_P, F_U \hookrightarrow F_\wp$.  
So all the maps $\phi_P$, $\phi_U$, and $\phi_\wp$ together define a homomorphism $\phi$ from $A$ to the inverse limit of the finite inverse system consisting of the fields $F_P$ (for $P \in \mc P$), $F_U$ (for $U \in \mc U$), and $F_\wp$ (for $\wp \in \mc B$).  But by \cite{HH:FP}, Proposition~6.3, this inverse limit is just $F$, with respect to the inclusions of $F$ into the fields $F_P, F_U, F_\wp$.
The $F$-homomorphism $\phi : A \to F$ then defines an $F$-rational point on $H$; i.e.\ $H(F) \neq \varnothing$ as asserted.
\end{proof}

\begin{cor}
Let $G_1$ and $G_2$ be linear algebraic groups such that $G_1\times G_2$ is a rational connected linear algebraic group.
Then the assertions of Theorems~\ref{factorization}, \ref{smooth factorization}, \ref{general factorization} and~\ref{torsor_injective} 
hold for $G_1$ and $G_2$.
\end{cor}
\begin{proof}
By symmetry, it suffices to prove the statement for $G_1$. 
Theorems~\ref{factorization}, \ref{smooth factorization} and~\ref{general factorization} hold for $G_1$ by choosing a preimage in $G_1\times G_2$ of each given point of $G_1$; factoring in $G_1\times G_2$ by the respective theorems for that rational connected group; and then projecting the factorization to $G_1$.  Theorem~\ref{torsor_injective} holds for $G_1$ because it holds for $G_1\times G_2$ and because $G_1\times G_2$ acts 
transitively on the points of any $F$-variety for which $G_1$ does.
\end{proof}

\begin{rem} \label{homog rk}
In the special case that $G$ is a (connected) reductive linear algebraic group over $F$ and $H$ is a projective $F$-variety, 
the transitivity condition in the above theorem 
simplifies.  Specifically, it is equivalent to the \textit{a priori} weaker condition that the 
group $G(\bar F)$ acts transitively (in the classical sense) on the set $H(\bar F)$, where $\bar F$ is an algebraic closure of $F$.  

To see this, note that for any field extension $E$ of $F$, the stabilizer of an $E$-point of $H$ is a parabolic subgroup of $G$, 
by the projectivity of $H$.  By the hypotheses on $G$, parabolic subgroups are self-normalizing (\cite{Borel}, Theorem~11.16), 
hence distinct $E$-points have distinct stabilizers; and two such subgroups are conjugate under $G(E)$ if they are conjugate under 
$G(\bar E)$ (\cite{Borel}, Theorem~20.9(iii)).  But the transitivity of $G(\bar F)$ on $H(\bar F)$ yields the same over $\bar E$, 
implying the conjugacy of the stabilizers.  Therefore the stabilizers of any two $E$-points of $H$ are conjugate under $G(E)$, 
and the points are then in the same $G(E)$-orbit.
These extra hypotheses on $G$ and $H$ are in fact satisfied in the situations below where we apply the above theorem (see the 
proofs of Theorems~\ref{iso_patch} and~\ref{index hasse}); but we will not need to use this fact.
\end{rem}

%%%%%%%%%%%%%%%%%%%%%%%%%%%%%%%%%%%%%%%%%%%%%%%%%%%%%%%%%%%%%%%%%%%%%%%%%%%%%%%%
%Section 4
\section{Quadratic forms} \label{quadratic}
%%%%%%%%%%%%%%%%%%%%%%%%%%%%%%%%%%%%%%%%%%%%%%%%%%%%%%%%%%%%%%%%%%%%%%%%%%%%%%%%

In this section we prove our results on quadratic forms.  We do this by reducing to a local problem, 
using the local-global principle in Theorem~\ref{torsor_injective}. For generalities concerning quadratic forms, 
we refer the reader to \cite{Lam} and~\cite{Grove}. 

Let $F$ be a field of characteristic unequal to $2$.
Recall that by the Witt decomposition theorem (\cite{Lam}, I.4.1), every quadratic form $q$ over $F$ 
may be decomposed as an orthogonal sum 
$q_t \perp q_a \perp q_h$, where $q_t$ is totally isotropic, $q_a$ is anisotropic, and $q_h$ is hyperbolic (or zero).  All factors 
are uniquely determined up to isometry. Here $q_r := q_a \perp q_h$ is regular (i.e.\ non-degenerate); and the $q_t$ factor does not occur if $q$ is regular.  The \textit{Witt index} $i_W(q)$ of $q$ is  $\frac12 \dim q_h$; if $q$ is regular this is the same as the dimension of any maximal 
totally isotropic subspace (\cite{Grove}, pp.~41-42).  Since
$\operatorname{char}(F)\neq 2$, every quadratic form over $F$ is isometric to
a diagonal form $a_1x_1^2+\cdots +a_nx_n^2$, which is denoted by
$\left<a_1,\ldots,a_n\right>$. If $E$ is a field containing $F$, then $q_E$
denotes the form $q$ viewed as a quadratic form over~$E$.

\begin{rem}\label{orthogonal rational}
If $q$ is a regular quadratic form over a field $F$ of characteristic unequal
to~$2$, then the special orthogonal 
group $\operatorname{SO}(q)$ of isometries of $q$ of 
determinant~$1$ is a rational connected linear algebraic group. 
More generally, let $A$ be a finite dimensional central simple 
$F$-algebra with an
involution~$\iota$ (i.e.\ an anti-automorphism of order $2$), and let $G  = \{a \in A^\times \,|\, \iota(a)=a^{-1}\}$.  
Then the classical
Cayley map $a \mapsto (1-a)(1+a)^{-1}$ defines a birational isomorphism from the
connected component $G^\circ$ to the set of $\iota$-skew symmetric elements, which
is an $F$-linear subspace of $A$ (\cite{BofInv}, p.~201,
Exercise~9); thus $G^\circ$ is rational.  \end{rem}

\begin{thm} \label{iso_patch}
In the context of Notation~\ref{P1 notation}, suppose $q$ is a quadratic form over $F$ of dimension unequal to two, such that $q_{F_\xi}$ is isotropic for each $\xi \in \mc P \cup \mc U$. Then $q$ is isotropic.
\end{thm}

\begin{proof}
If the dimension of $q$ is one, then each $q_{F_\xi}$ is totally isotropic, and hence not regular.  Thus neither is $q$, and the conclusion follows in this case.

Now suppose $n := \dim q \ge 3$.  By Witt decomposition, we may write $q = q_t \perp q_r$, where $q_r$ is regular and $q_t$ is totally isotropic. If $q_t \neq 0$ then $q$ is isotropic and there is nothing to show.  Therefore, we may assume that $q$ (and hence each $q_{F_\xi}$) is regular.  

Let $H$ be the projective quadric hypersurface defined by $q$. 
Observe that $\Orth(q)$ acts transitively on the points of the $F$-variety
$H$ (see the definition before Theorem~\ref{torsor_injective}).  To see this, let $L$ be a field extension of $F$, and let 
$\xi_1, \xi_2 \in H(L)$.  These points correspond to lines $W_1, W_2$ through the origin in $\mbb A_L^n$ that are totally isotropic with respect to $q_L$;
and hence any isomorphism $f:W_1 \to W_2$ as $L$-vector spaces is an isometry.  By Witt's extension theorem (\cite{Grove}, Theorem~5.2), 
such an $f$ extends to an isometry of $\mbb A_L^n$ taking $W_1$ to $W_2$.  That is, some element of the orthogonal group 
$\Orth(q)(L)$ carries $\xi_1$ to $\xi_2$. Hence $\Orth(q)$ acts transitively on the points of $H$.

Since $n \ge 3$, the quadric hypersurface $H$ is connected.  Therefore, the special orthogonal group $\SOrth(q)$, which is the connected component of $\Orth(q)$, also acts transitively on the points of 
$H$.  By Remark~\ref{orthogonal rational}, the group $\SOrth(q)$ is rational.  Since $\SOrth(q)$ acts transitively on the points of
$H$, Theorem~\ref{torsor_injective} implies that $H(F)$ is non-empty provided that each $H(F_\xi)$ is.  That is, if each $q_{F_\xi}$ is isotropic then so is $q$.
\end{proof}

We note that in the above proof, the transitivity of $\SOrth(q)$ on the points of $H$ can also be proven by applying Remark~\ref{homog rk}.  
Namely, $\SOrth(q)$ is connected and reductive, and the projective variety $H$ is homogeneous for that group over $\bar F$ 
(i.e., $\SOrth(q)(\bar F)$ acts transitively on $H(\bar F)$).  So Remark~\ref{homog rk} implies that $\SOrth(q)$ acts 
transitively on the points of $H$.

The above result can be regarded as a Hasse-Minkowski theorem for quadratic forms over the function field of a curve defined 
over a complete discretely valued field.  As a consequence, we obtain the
following:

\begin{cor} \label{witt_index}
In the context of Notation~\ref{P1 notation}, suppose $q$ is a regular
quadratic form over $F$.
Then $i_W(q)\in \{\operatorname{min}(i_W(q_{F_\xi})),
\operatorname{min}(i_W(q_{F_\xi}))-1\}$, where the minimum is taken over all
$\xi \in  \mc P \cup \mc U$.
Moreover the second case can occur only if all $q_{F_\xi}$ are hyperbolic.\end{cor}

\begin{proof}
We proceed by induction.  If the dimension of $q$ is one, then $q$ and 
$q_{F_\xi}$ cannot contain a hyperbolic plane, and so the Witt indices are all $0$.  If the dimension of $q$ is two, and if any $q_{F_\xi}$ is anisotropic, then so is $q$; thus $i_W(q)=0=\operatorname{min}(i_W(q_{F_\xi}))$. The remaining two-dimensional case is when all $q_{F_\xi}$ are hyperbolic, in which case $\operatorname{min}(i_W(q_{F_\xi})) = 1$ and 
$i_W(q)$ is equal to $1$  or $0$ depending on whether or not $q$ is hyperbolic.

For the inductive step, consider a form $q$ of dimension $n \ge 3$ and assume
that the assertion holds for forms of dimension $n-2$.  We may suppose that
$\operatorname{min}(i_W(q_{F_\xi}))$ is nonzero (otherwise there is nothing to
show). In particular, each $q_{F_\xi}$ is isotropic. Then by
Theorem~\ref{iso_patch}, $q$ is isotropic.
By Witt decomposition, this
implies that $q\simeq h\perp q'$ for some $q'$ and a hyperbolic plane $h$.
Hence $i_W(q')=i_W(q)-1$ and $\dim(q')=n-2$. Moreover $q_{F_\xi}\simeq
h_{F_\xi}\perp q'_{F_\xi}$ for all $\xi$.
Thus $q_{F_\xi}$ is hyperbolic if and only if $q'_{F_\xi}$ is, and 
$i_W(q'_{F_\xi})=i_W(q_{F_\xi})-1$.  The conclusion of the corollary thus holds for $q'$ and hence for $q$.
\end{proof}

We thank J.-L.~Colliot-Th\'el\`ene for bringing to our attention the following example, which shows that Theorem~\ref{iso_patch} does not in general hold in dimension two, and that the second case of Corollary~\ref{witt_index} can occur for forms that are hyperbolic over the fields $F_\xi$.

\begin{ex}
Let $T$ be a complete discrete valuation ring with uniformizer $t$,
fraction field $K$, and residue field $k$ of characteristic unequal to
$2$.  Consider the field $F = K(x)[y]/(y^2-x(x-1)(1-xt))$, which is a degree two extension of the function field $K(x)$ of $\PP^1_T$.  The normalization $\wh X$ of $\PP^1_T$ in $F$ is a normal projective $T$-curve that is equipped with a degree two finite morphism $f:\wh X \to \PP^1_T$.
The closed fiber $X$ of $\wh X$ is a rational $k$-curve with a single node
$P$, which is the unique point lying over the point at infinity on
$\PP^1_k$; and the complement $U$ of $P$ in $X$ is the inverse image of
the affine $k$-line.  The general fiber of $\wh X$ is an elliptic curve
over $K$; this is a Tate curve in the case that $K$ is a $p$-adic field.
  With $a=x(x-1)$, let $\wh X' \to \wh X$ be the unramified degree two
cover with function field $F' := F[\sqrt{a}]$, and let $q$ be the
quadratic form $\left< a,-1 \right>$ over $F$.  Then $q$ is anisotropic
over $F$ because $a$ is not a square in $F$.  But $\wh X' \to \wh X$ is
split over $P$ and $U$ and hence over the spectra of $\wh R_P$ and $\wh
R_U$.  Hence the two-dimensional form $q$ becomes isotropic (and thus
hyperbolic) over $F_P$ and over $F_U$.  This shows that
Theorem~\ref{iso_patch} does not always hold for forms of dimension two.
  Moreover, the Witt indices $i_W(q_P)$ and $i_W(q_U)$ are equal to one,
whereas that of the anisotropic form $q$ is equal to zero.  Thus
$i_W(q)$ can equal $\operatorname{min}(i_W(q_{F_\xi}))-1$ in the locally
hyperbolic case of Corollary~\ref{witt_index}.
\end{ex}

Next, we consider a variant on Hensel's Lemma.

\begin{lem} \label{affine_hensel}
Let $R$ be a ring and $I$ an ideal such that $R$ is $I$-adically
complete.  Let $X$ be an affine $R$-scheme with structure morphism $\phi
: X \to \Spec R$. Let $n \ge 0$.  If $s_n : \Spec R/I^n \to X \times_
R (R/I^n)$ is a section of $\phi_n := \phi \times_R (R/I^n)$ and its image lies
in the smooth locus of $\phi$, then $s_n$ may be extended to a section of $\phi$.
\end{lem}

\begin{proof}
Write $X = \Spec S$ for some $R$-algebra $S$, with structure map $i:R \to S$. Let $X' \subseteq X$ be the smooth locus of $X$ over $R$, and let $\phi'$ be the restriction of $\phi$ to $X'$.  Since $X'$ is smooth over $R$, it is formally smooth over $R$ (see \cite{EGAIV4}, Definition~17.3.1).  That is, for any $m \ge 1$, any section $s_m: \Spec R/I^m \to X' \times_R (R/I^m)$ of $\phi_m':= \phi' \times_R (R/I^m)$ lifts to a section 
$s_{m+1}: \Spec R/I^{m+1} \to X' \times_R (R/I^{m+1})$ of $\phi_{m+1}':= \phi \times_R (R/I^{m+1})$ (see~\cite{EGAIV4}, Definition~17.1.1).  Hence by induction, there is a compatible system of sections
$s_m: \Spec R/I^m \to X' \times_R (R/I^m) \subseteq X \times_R (R/I^m)$ of the maps $\phi_m$,  for $m \ge n$, with each $s_m$ in particular lifting $s_n$.  Here the morphism $s_m: \Spec R/I^m \to X \times_R (R/I^m)$ corresponds to a retract $\pi_m:S/I^mS \to R/I^m$
of the mod $I^m$-reduction $i_m:R/I^m \to S/I^mS$ of $i$ (i.e., $\pi_m
\circ i_m$ is the identity on $R/I^m$).  Writing $p_m:S \to S/I^mS$ for the reduction modulo $I^mS$, we obtain a compatible system of maps
$\pi_m \circ p_m:S \to R/I^m$, which in turn defines a map $\pi:S \to R$ given by their inverse limit (using that $R$ is $I$-adically complete).  The map $\pi$ is then a retract of $i$ and thus corresponds to a section of $\phi$ that extends $s_n$.
\end{proof}

In fact, the above lemma holds even without the assumption that $X$ is affine over $\Spec R$, by using Corollaire~5.1.8 and 
Th\'eor\`eme~5.4.1 of \cite{EGAIII} in place of the inverse system argument at
the end of the above proof.  By citing the above lemma in that more general
form, one could use the projective hypersurface $H$ of Theorem~\ref{iso_patch}
rather than the associated affine quadric $Q$ in the proof of
Proposition~\ref{u-inv patches} below, and one would not need to choose an
affine open subset in the proof of Proposition~\ref{csa patches}. The proof
above, however, is more elementary.

\medskip

In the context of Notation~\ref{P1 notation}, assume that the residue field $k$ of $T$ has characteristic unequal to $2$.  In particular, $F$ does not have characteristic $2$.  As a consequence, any quadratic form $q$ over $F$ may be diagonalized.  

\begin{dfn} \label{discrimdef}
If $q = \left<a_1, \ldots, a_n\right>$ is a regular diagonal quadratic form over a field $F$ as above, its \textit{singular divisor} on $\wh X$ is the sum of those prime divisors on $\wh X$ at which the divisor of some $a_i$ (viewed as a rational function on $\wh X$) has odd multiplicity.  
\end{dfn}

Observe in the above definition that a change of variables $x_i' = c_ix_i$
with $c_i \in F^\times$ does not affect the singular divisor, since each $a_i$
is then multiplied by a square. In particular, in the context of
Notation~\ref{P1 notation}, 
for every $\xi \in \mc U \cup \mc P$, there is such a change of variables
taking $q$ to another 
diagonal form $q' = \left<a_1', \ldots, a_n'\right>$ with each $a_i' \in \wh
R_\xi \cap F$.   
Here $q'$ is isometric to the form $q$, and has the same singular divisor.  

We recall the following standard result:

\begin{lem} \label{resolution}
Let $S$ be a two-dimensional excellent normal scheme.  Then there is a birational morphism $\pi:S' \to S$ such that $S'$ is regular.  Moreover if $D$ is a divisor on $S$ then we may choose $\pi:S' \to S$ such that the support of $\pi^{-1}(D)$ has only normal crossings.
\end{lem}

\begin{proof}
The first part of the assertion is resolution of singularities for surfaces; see \cite{Abh} or \cite{Lip}.  If $\pi:S' \to S$ is as in the first part, then by \cite{Lip}, page 193, there is a birational morphism $S'' \to S'$ of regular surfaces for which the inverse image of $D' = \pi^{-1}(D)$ is a normal crossing divisor on $S''$.
\end{proof}

Recall from Definition~\ref{uinv} that the $u$-invariant of a field is the
maximal dimension of anisotropic quadratic forms over the field.  Below we use Notation~\ref{notation}. 

\begin{prop} \label{u-inv patches}
Let $\wh X$ be a regular projective $T$-curve with function field $F$ and closed fiber $X$.  Let $q$ be a regular diagonal quadratic form over $F$.   
\renewcommand{\theenumi}{\alph{enumi}}
\begin{enumerate}
\item  \label{big_local}
Let $X_0$ be an irreducible component of $X$, with function field $\kappa(X_0)$.
If $\dim q > 2u(\kappa(X_0))$
then $q_{F_U}$ is isotropic for some Zariski dense affine open subset $U \subset X_0$.
\item \label{small_local}
Let $P$ be a closed point of $X$ with residue field $\kappa(P)$, and assume that there are $c$ components of the singular divisor of $q$ that pass through $P$.  If $\dim q > 2^c u(\kappa(P))$ then $q_{F_P}$ is isotropic.
\end{enumerate}
\end{prop}

\begin{proof}
Write $q = \left<a_1, \ldots, a_d \right>$ with $a_i \in F$.  After a multiplicative change of variables, we may assume that each $a_i$ lies in $\wh R_U$ or $\wh R_P$ respectively.

(\ref{big_local}) Since $\wh X$ is regular, the maximal ideal of the local
ring at the generic point $\eta$ of $X_0$ (which is a codimension one point of $\wh X$) is principal.  So there is a Zariski affine open neighborhood $\Spec R \subset \wh X$ of $\eta$ whose closed fiber $U$ is an affine open subset of $X_0$ along which $X$ is regular, and such that the defining ideal of $U$ in $\Spec R$ is principal, say with generator $t_0 \in R \subset F$.  Thus $t \in t_0 R$.

Consider the principal divisor $(a_1\cdots a_d)$ on $\wh X$.  Each of its components other than $X_0$ (which may or may not be a component of this divisor) meets $U$ at only finitely many points.  After shrinking $U$ by deleting those points, we may assume that the restriction of this divisor to $\Spec \wh R_U$ is 
either trivial or is supported along the closed fiber.  Thus $q$ is isometric
to a diagonal form over $\wh R_U \cap F$ whose entries are each either units
in $\wh R_U$ or the product of $t_0$ and a unit (since even powers of $t_0$ may be factored out).  Therefore, over $\wh R_U \cap F$, the form $q$ is isometric to $q' \perp t_0 q''$, where $q',q''$ are diagonal forms all of whose entries are units in $\wh R_U$.  It suffices to show (possibly after shrinking $U$ again) that either $q'$ or $q''$ is isotropic over $F_U$, since then $q' \perp t_0 q''$ and hence $q$ would be as well.

Since $\dim q = \dim q' + \dim q''$, 
the assumption on $\dim q$ implies that either $q'$ or $q''$ has dimension $e$ greater than $u(\kappa(X_0))$.  Let 
$Q\subset {\mathbb A}^e_{\wh R_U}$ be the affine quadric cone defined by that subform and let $Q' \subset Q$ be the complement of the origin.
(Thus $Q$ is the affine cone over the projective quadric defined by that subform.)  
Since $e>u(\kappa(X_0))$, it follows that $Q'(\kappa(X_0)) \neq \varnothing$.
Therefore there is a rational section of the affine morphism $Q \to \Spec R_U$ over $U$ whose image lies on (the closed fiber of) $Q'$.  
This rational section is defined as a morphism on a dense open subset of $U$.  Replacing $U$ by that subset, we may 
assume that this rational map is a section $U \to Q' \subset Q$
of $Q \to \Spec \wh R_U$ over $U$.  Now $Q'$ is the smooth locus of $Q$ over $\wh R_U$, since the residue characteristic of $T$ is not $2$, 
and since the quadratic form is diagonal with unit coefficients.  So by Lemma~\ref{affine_hensel},  the section over $U$ lifts to a 
section $\Spec \wh R_U \to Q$.  This yields an $F_U$-point of $Q$ that is not the origin (since its restriction to $U$ is not).  
Hence either $q'$ or $q''$ is isotropic over $F_U$, as desired.

(\ref{small_local}) Since $\wh R_P$
is regular and local, it is a unique factorization domain
(\cite{Eis:CA}, Theorem 19.19).  So the components $D_j$ of the singular
divisor $D$ that pass through $P$ are the loci of irreducible elements $r_j
\in \wh R_P$, $1 \le j \le c$.   After rescaling the variables we obtain an
isometric form $q' = \left<a_1', \ldots, a_d' \right>$ with the same singular
divisor as $q$, such that each $a_i'$ is of the form $u_i\prod r_j^{n_{ij}} \in \wh R_P \cap F$ for some units $u_i \in \wh R_P^\times$, where each $n_{ij} = 0$ or $1$.  For each  
$c$-tuple $\lambda = (\lambda_1, \ldots,
\lambda_c) \in \{0, 1\}^c$, 
let $S(\lambda) = \{i \,|\, n_{i,j}
= \lambda_j \ \text{for}\  j=1,\dots,c\}$
and define $q_\lambda = \underset{i \in S(\lambda)}{\perp} u_i$.
Let $\ov q_\lambda$ be the
reduction of $q_\lambda$ modulo the maximal ideal of $\wh R_P$.

Since $\dim q > 2^c u(\kappa(P))$, at least one of the $2^c$ forms $\ov q_\lambda$ over $\kappa(P)$ has dimension $e$ 
greater than $u(\kappa(P))$.  
Hence $Q'(\kappa(X_0)) \neq \varnothing$, where $Q'$ is the complement of the origin in the affine cone $Q\subset {\mathbb A}^e_{\wh R_P}$ 
defined by the form $q_\lambda$.  Since $Q'$ is the smooth locus of $Q$ over $\wh R_P$, Lemma~\ref{affine_hensel} lifts the 
$\kappa(X_0)$-point of $Q'$ to a section $\Spec \wh R_P \to Q$.  
This yields an $F_P$-point of $Q$ that is not the origin, thereby showing that $q_\lambda$ is isotropic over $F_P$.  
Thus so is $(\prod r_j^{\lambda_j}) q_\lambda$.
Since $(\prod r_j^{\lambda_j}) q_\lambda$ is a subform of $q'$, this implies that $q'$ is isotropic as well.  Hence so is the isometric form $q$.
\end{proof}

\begin{lem} \label{uinvdvr}
Let $T$ be a discrete valuation ring with fraction field~$K$ and residue field~$k$ of characteristic unequal to~$2$.  
Then $u(K) \ge 2u(k)$ and $u_s(K) \ge 2u_s(k)$.
\end{lem}

\begin{proof}
Let $t$ be a uniformizer of $T$, and hence of its completion $\wh T$.  Let $q$
be an anisotropic form over $k$ and let $n$ be its dimension.  Since the
characteristic of $k$ is not $2$, 
we may assume that $q$ is diagonal.  Let $\til q$ be a diagonal lift of $q$ to $T$. 
By \cite{Lam}, VI.1.9(2), $q' = \til q \perp t \til q$ is anisotropic over $\wh K$, the fraction field of $\wh T$.  Hence $q'$ is also anisotropic over $K$.  This shows that if $u(k)\ge n$ then $u(K) \ge 2n$; and that proves the first assertion.

For the second assertion, let $n=u_s(k) \in {\frac 12}\Z$.  By definition of $u_s$, there is either an anisotropic quadratic form $q$ of dimension $n \in \Z$ over a finite extension $E$ of $k$, or an anisotropic quadratic form $q$ of dimension $2n \in \Z$ over a finitely generated field extension $E$ of $k$ of transcendence degree one.  After replacing $q$ by an isometric form, we may assume in either situation that $q$ is diagonal. We consider the above two cases in turn.

In the former case, $u(E)=n$.  Observe that there is a finite extension $L$ of
$K$ whose residue field is $E$.  (Namely, we inductively reduce to the case that $E = k[a]$ for some $a \in E$, say with monic minimal polynomial $f(y) \in k[y]$; and then take $L = K[\til a]$, where $\til a$ is a root of some monic lift of $f(y)$ to $T[y]$.)  By the first  assertion of the lemma, $u(L) \ge 2 u(E) = 2n$.  But $u_s(K) \ge u(L)$.  So $u_s(K) \ge 2n = 2u_s(k)$.

In the latter case, $u(E) = 2n$.  Let $\{x\}$ be a transcendence basis for $E$
over $k$. We may assume that $E = k(x)[a]$ for some $a \in E$, say with monic minimal polynomial $f \in k[x,y]$ over $k(x)$.  Take a monic lift $\til f \in T[x,y]$ of $f$ and let $F$ be the fraction field of $T[x,\til a]$, where $\til a$ is a root of $\til f$.  This is a field of transcendence degree one over $K$.
Taking the normalization of $T[x]$ in $F$, we obtain a normal $T$-curve $\til X$ whose closed fiber $X$ is irreducible and has function field $E$.  Let $\xi$ be the generic point of $X$, and let $R$ be the local ring of $\til X$ at $\xi$.  Thus $R$ is a discrete valuation ring with fraction field $F$ and residue field $E$.  By the first assertion of the lemma, $u(F) \ge  2u(E) = 4n$; and so $u_s(K) \ge 2n = 2u_s(k)$.
\end{proof}
 
We now prove our main result about quadratic forms, in terms of the strong $u$-invariant (see Definition~\ref{suinv}). 

\begin{thm} \label{main}
Let $T$ be a complete discrete valuation ring having fraction field $K$ and residue field $k$, with $\cha k \ne 2$.  Then $u_s(K) = 2u_s(k)$.
\end{thm}

\begin{proof}

By the second part of Lemma~\ref{uinvdvr}, $u_s(K) \ge 2u_s(k)$.  It remains to
show that $u_s(K) \le 2u_s(k)$.  Write $n=u_s(k)$, so
every finite extension of $k$ has $u$-invariant at most $n$.  By Springer's theorem on nondyadic complete discrete valuation fields (see \cite{Lam}, VI.1.10 and XI.6.2(7)), every finite extension of $K$ has $u$-invariant at most $2n$. To
prove the desired inequality, we must therefore show that every finitely generated
field extension of transcendence degree one over $K$ has $u$-invariant at
most $4n$. Let $F$ be such a field extension, and let $q$ be a quadratic form over $F$ of dimension $> 4n$.  We wish to show that $q$ is isotropic.  

We may assume that $q$ is regular, since otherwise it is trivially isotropic.  The characteristic of $F$ is not $2$, by the same property for $k$; so there is a diagonal form over $F$ that is isometric to $q$, and we may replace $q$ by that form.  Let $\wh X_1$ be a normal projective model for $F$ over $T$, and let $D_1$ be the singular divisor of $q$ on $\wh X_1$ (see Definition~\ref{discrimdef}).  By Lemma~\ref{resolution}, there is a regular projective $T$-curve~$\wh X$ with function field $F$, and a birational morphism $\pi:\wh X \to \wh X_1$, such that $\pi^{-1}(D_1)$ has only normal crossings.  The singular divisor $D$ of $q$ on $\wh X$ is contained in $\pi^{-1}(D_1)$, and so it also has only normal crossings. 

For each irreducible component $X_0$ of the closed fiber $X$ of $\wh X$, the function field $\kappa(X_0)$ has transcendence degree one over $k$; and so $u(\kappa(X_0)) \le 2u_s(k) = 2n$ by the definition of $u_s$.  Hence for each such component, $\dim q > 4n \ge 2u(\kappa(X_0))$; and thus by Proposition~\ref{u-inv patches}(\ref{big_local}), we may pick a Zariski dense affine open subset $U_0 \subset X_0$ such that $q_{F_{U_0}}$ is isotropic.
By~\cite{HH:FP}, Proposition~6.6, there is a finite morphism $f:\wh X \to \PP^1_T$ such that $f^{-1}(\infty)$ contains the (finitely many) points of $X$ that do not lie in any of our chosen sets $U_0$ (as $X_0$ ranges over the components of $X$), as well as containing all the closed points at which distinct components of $X$ meet.  Under Notation~\ref{P1 notation}, and by the choice of $f$,
each $U \in \mc U$ is contained in one of the above sets $U_0$; hence 
$F_U$ contains $F_{U_0}$.  Thus 
$q_{F_U}$ is isotropic for each $U \in \mc U$.  Meanwhile, since the singular divisor of $q$ has at most normal crossings, the number of components of this divisor that pass through any closed point $P \in X$ is at most two.  Since $u(\kappa(P)) \le u_s(k) = n$ for each $P$, we have that $\dim q > 4n \ge 4u(\kappa(P)) \ge 4$.  So by Proposition~\ref{u-inv patches}(\ref{small_local}), $q_{F_P}$ is isotropic for each $P \in \mc P$. Therefore by Theorem~\ref{iso_patch}, $q$ is indeed isotropic.
\end{proof}

The above result generalizes from the complete case to the henselian case.  First we prove a lemma that relies on the Artin Approximation Theorem (\cite{Artin}, Theorem~1.10).

\begin{lem}\label{approx}
Let $T$ be an excellent henselian discrete valuation ring, and let $\wh K$ be the completion of its fraction field $K$.  
Let $E$ be a finitely generated field extension of $K$ having transcendence degree at most one, and let $X$ be a projective $E$-variety.  
Suppose that 
$X(\wh E) \ne \varnothing$ for every finitely generated field extension $\wh E$ of $\wh K$ that contains $E$ and satisfies 
$\td_{\wh K} \wh E = \td_K E$. 
Then $X(E) \ne \varnothing$.
\end{lem}

\begin{proof}
Let $t$ be a uniformizer of $T$.  By hypothesis, $X$ is a closed subset of some $\PP^{\,n}_E$ defined by homogeneous polynomials 
$f_1,\dots,f_m \in E[z_0,\dots,z_n]$.  

First consider the case that $E$ is finite over $K$.  After multiplying the polynomials $f_i$ by some power of $t$, we may assume that 
each $f_i$ lies in $S[z_0,\dots,z_n]$, where $S$ is the integral closure of $T$ in $E$ (this being the valuation ring of $E$).
Extend the valuation on $K$ to $E$.  Then the completion $\wh E$ of $E$ is finite over $\wh K$ (and is the compositum of its subfields 
$\wh K$ and $E$); so by assumption, $X$ has an $\wh E$-point.  After multiplying a choice of coordinates of the point by some power of $t$, 
we may assume that each coordinate $\bar a_i$ lies in the valuation ring $\wh S$ of $\wh E$ (where $\wh S$ is also the integral closure of 
$\wh T$ in $\wh E$).  Thus we have a 
solution $(\bar a_0,\dots,\bar a_n) \in \wh S^{n+1}$ of the polynomial equations $f_1=\dots=f_m=0$, with not all $\bar a_i$ equal to~$0$.  
So for some $e>0$ and some $i_0$, the element
$\bar a_{i_0} \in \wh S$ is not congruent to zero modulo $t^e\wh S$.  By the Artin Approximation Theorem (\cite{Artin}, 
Theorem~1.10), there exists a solution $(a_0,\dots,a_n) \in S^{n+1}$ to the system $f_1=\dots=f_m=0$ such that $a_i \equiv \bar a_i$ modulo
$t^e\wh S$. In particular,  $a_{i_0} \ne 0$.  Hence $(a_0,\dots,a_n)$ defines an $S$-point of $X$, and $X(E) \ne \varnothing$. 

It remains to consider the case that $E$ has transcendence degree one over $K$.  
Thus $E = K(x)[y_1,\dots,y_r]/(g_1,\dots,g_s)$, a finite extension of $K(x)$, for some polynomials $g_i \in T[x,y]$ 
defining a prime ideal $I \subset T[x,y]$ that does not extend to the unit ideal in $K(x)[y]$.  (Here for short we write $y$ for $y_1,\dots,y_r$.  
Below we also write $g$ for $g_1,\dots,g_s$.)  Since $\wh K(x)[y]$ is faithfully flat over $K(x)[y]$, the extension $\wh I$ of $I$ to 
$\wh T[x,y]$ does not induce the unit ideal in $\wh K(x)[y]$.  In particular, $\wh I$ is a proper ideal  in $\wh T[x,y]$. 

We claim that 
$\wh I$ is a prime ideal in $\wh T[x,y]$.  For if it were not, then there would exist $c,d \in \wh T[x,y]\smallsetminus \wh I$ 
for which $cd\in \wh I$; i.e., $cd= \sum e_ig_i$ for some $e_i \in \wh T[x,y]$. 
But then \cite{Artin}, Theorem~1.10, applied to the coefficients of the elements $c,d,e_i$, would produce a contradiction to $I$ being prime, 
which proves the claim.  

Since $\wh I = (g)$ is prime in $\wh T[x,y]$, 
the ring $\wh E := \wh K(x)[y]/(g) = E \otimes_K \wh K$ is a domain. 
But $\wh E$ is finite over the field $\wh K(x)$, since $E$ is finite over $K(x)$; hence $\wh E$ is a field, and is the compositum of 
its subfields $E$ and $\wh K$.  Since $\wh E$ has transcendence degree one over $\wh K$, by the hypothesis there is an $\wh E$-point of $X$; 
i.e.\ a solution $(\bar a_0,\dots,\bar a_n) \in \wh E^{n+1}$ to the system $f_1=\dots=f_m=0$, with 
some $\bar a_{i_0} \ne 0$.
Lifting each $\bar a_i$ to an element of $\wh K(x)[y]$ and then multiplying by a non-zero element of $\wh T[x]$, we obtain elements 
$\hat a_i \in \wh T[x,y]$ for $i=0,\dots,n$, and elements $b_{jh} \in \wh T[x,y]$ for $j=1,\dots,m$ and $h=1,\dots,s$, such that 
$f_j(\hat a_0,\dots,\hat a_n) = \sum_h b_{jh} g_h \in \wh T[x,y]$ for all $j$.  Moreover $\hat a_{i_0} \not\in \wh I \subset\wh T[x,y]$,
and hence for some $e>0$
its image in $(\wh T/t^e \wh T)[x,y]$ does not lie in the reduction of $\wh I$.  Applying \cite{Artin}, 
Theorem~1.10, to the coefficients of $x$ and $y$ in 
$\hat a_i,b_{jh}$, there exist $a_i',b_{jh}' \in T[x,y]$ that are congruent to $\hat a_i, b_{jh}$ modulo $t^e$, such that 
$f_j(a_0',\dots,a_n') = \sum_h b_{jh}'g_h \in T[x,y]$ for all $j$.  The reductions of $a_0',\dots,a_n'$ modulo $I$ then yield a solution 
$(a_0,\dots,a_n) \in (T[x,y]/I)^{n+1} \subset E^{n+1}$ to the system $f_1=\dots=f_m=0$, with $a_{i_0} \ne 0$.  This solution then defines 
an $E$-point of $X$.
\end{proof}

\begin{cor}[\label{hensel}]
Let $T$ be an excellent henselian discrete valuation ring with fraction field $K$ and with residue field $k$ of characteristic unequal to $2$.  Then $u_s(K) = 2u_s(k)$.
\end{cor}

\begin{proof}
Let $\wh K$ be the completion of $K$; this is a complete discretely valued field with residue field $k$.  Thus $u_s(\wh K) = 2u_s(k)$ by Theorem~\ref{main}.  Also, $u_s(K) \ge 2u_s(k)$ by the second part of Lemma~\ref{uinvdvr}.  Thus to prove the result it suffices to show that $u_s(K) \le 2u_s(k)$.   

So let $E$ be a finitely generated field extension of $K$ having transcendence degree $\ell \le 1$, and let $q$ be a quadratic form over $E$ of dimension $n>2^{1+\ell}u_s(k)$.  We wish to show that $q$ is isotropic over $E$.  
Let $H$ be the hypersurface in $\PP^{\,n-1}_E$ defined by $q$ (as in the proof of Theorem~\ref{iso_patch}).  Now $u_s(\wh K) = 2u_s(k)$,
and so
$n>2^\ell u_s(\wh K)$.  
Hence over every finitely generated field extension of $\wh K$ having transcendence degree $\ell$, 
over which $q$ is defined (e.g.\ containing $E$), the form $q$ is isotropic.  
Equivalently, $H$ has a rational point over each such field.  So by
Lemma~\ref{approx}, 
$H$ has a rational point over $E$; i.e.\ $q$ is isotropic over $E$.  
\end{proof}
 
Recall that $k$ is a \textit{$C_d$-field} if for all $m \ge 1$ and $n > m^d$, every homogeneous form of degree $m$ in $n$ variables over $k$ has a non-trivial solution in $k$.  In particular, a $C_d$-field $k$ satisfies $u(k) \le 2^d$ (by taking $m=2$).  Moreover, every finite extension of $k$ is also a $C_d$-field, and every one-variable function field over $k$ is a $C_{d+1}$-field (\cite{Serre:CG}, II.4.5).  Hence $u_s(k) \le 2^d$ for a $C_d$-field $k$.  

Recall also that a field $K$ is called an \textit{$m$-local field} with
\textit{residue field $k$} if there is a sequence of fields $k_0,\dots,k_m$
with $k_0=k$ and $k_m=K$, and such that $k_i$ is the fraction field of an
excellent henselian discrete valuation ring with residue field $k_{i-1}$ for
$i=1,\dots,m$.  For $K$ and $k$ as above, it follows by induction that a
finite extension of $K$ is an $m$-local field whose residue field is a finite
extension of $k$.  
Also note that if $\operatorname{char}(k)\neq 2$, 
$u_s(K) = 2^m u_s(k)$ by Theorem~\ref{hensel} and induction; and so $u(F) \le 2^{m+1}u_s(k)$ for any one-variable function field $F$ over $K$, by definition of $u_s$.

\begin{cor}\label{cor_mlocal}
Suppose that $K$ is an $m$-local field whose residue field $k$ is a $C_d$-field of characteristic unequal to $2$, and let $F$ be a function field over $K$ in one variable.  
\renewcommand{\theenumi}{\alph{enumi}}
\begin{enumerate}
\item \label{mlocalineq}
Then $u_s(K) \le 2^{d+m}$ and hence $u(F) \le 2^{d+m+1}$.  
\item \label{mlocalcdK} 
If $u(k)=2^d$ then $u(K)=2^{d+m}$. Moreover if some normal $K$-curve with function field $F$ has a $K$-point, then $u(F)=2^{d+m+1}$. 
\item \label{mlocalcdF}
If $u(k')=2^d$ for every finite extension $k'/k$, then $u(F)=2^{d+m+1}$. 
\end{enumerate}
\end{cor}

\begin{proof}
(\ref{mlocalineq}) By the discussion preceding this result, 
$u_s(K) = 2^m u_s(k)$ and
$u(F) \le 2^{m+1}u_s(k)$.  But $u_s(k) \le 2^d$ since $k$ is a $C_d$-field.  So the conclusion follows.

(\ref{mlocalcdK}) Since $k$ is a $C_d$-field with $u(k)=2^d$, we have that $u(k) \le u_s(k) \le 2^d$ and so in fact all three quantities are equal.  Applying Lemma~\ref{uinvdvr} and induction yields that $u(K)
\ge 2^m u(k) = 2^{d+m}$.  But $u(K) \le u_s(K) \le 2^{d+m}$ by (\ref{mlocalineq}).  So all these quantities are equal too, proving the first assertion.

Now let $X$ be a normal $K$-curve with function field $F$ and let $\xi$ be a $K$-point on $X$.  The local ring at $\xi$ has fraction field $F$ and residue field $K$.  So Lemma~\ref{uinvdvr} implies that $u(F) \ge 2 u(K) =2^{d+m+1}$.  The reverse inequality follows from part~(\ref{mlocalineq}).

(\ref{mlocalcdF}) 
The inequality $u(F) \le 2^{d+m+1}$ is given in part~(\ref{mlocalineq}).  To show the reverse inequality, choose a normal (or equivalently, regular) $K$-curve $X$ having function field $F$, and choose a closed point $\xi$ on $X$.  Let $R$ be the local ring of $X$ at $\xi$, with residue field $\kappa(\xi)$.  Then the fraction field of $R$ is $F$, and $\kappa(\xi)$ is a finite extension of $K$.  Hence $\kappa(\xi)$ is an $m$-local field whose residue field $k'$ is a finite extension of $k$.  By hypothesis, $u(k')=2^d$; and $k'$ is a $C_d$-field since $k$ is.  So applying part~(\ref{mlocalcdK}) to $k'$ and $\kappa(\xi)$, we find that $u(\kappa(\xi)) = 2^{d+m}$.  Lemma~\ref{uinvdvr} now implies that $u(F) \ge 2^{d+m+1}$.
\end{proof}

For example, if $k$ is a field of transcendence degree $d$ over an algebraically closed field of characteristic unequal to $2$, then $k$ is a $C_d$-field (theorem of Tsen-Lang, see \cite{Serre:CG}, II.4.5(b)).  So $u(F) \le 2^{d+m+1}$ for any one-variable function field $F$ over an $m$-local field with residue field $k$, by Corollary~\ref{cor_mlocal}(\ref{mlocalineq}).  This was known in the special case that $F$ is a one-variable function field over $k((t))$.  Namely, in that situation, $k((t))$ is a $C_{d+1}$-field by Theorem~2 of \cite{Gre}; so $F$ is a $C_{d+2}$-field by the theorem of Tsen-Lang cited above and hence $u(F) \le 2^{d+2}$.

As a special case of Corollary~\ref{cor_mlocal}(\ref{mlocalcdK}), the $u$-invariant of $K(x)$ equals $2^{d+m+1}$ if $K$ is an $m$-local field whose residue field is $C_d$, has $u$-invariant $2^d$, and does not have characteristic~$2$.

\begin{cor}\label{finac}
Let $F$ be a one-variable function field over an $m$-local field whose residue field $k$ has characteristic unequal to $2$.
\renewcommand{\theenumi}{\alph{enumi}}
\begin{enumerate}
\item  \label{mlocalac}
If $k$ is algebraically closed, then $u(F)=2^{m+1}$.
\item \label{mlocalfinite}
If $k$ is a finite field, then $u(F)=2^{m+2}$.
\end{enumerate}
\end{cor}

\begin{proof}
(\ref{mlocalac}) This is a special case of Corollary~\ref{cor_mlocal}(\ref{mlocalcdF}), using that an algebraically closed field $k$ is $C_0$, satisfies $u(k)=1$, and has no non-trivial finite extensions.

(\ref{mlocalfinite}) A finite field $k$ is $C_1$ (by \cite{Serre:CG}, II.3.3(a)), and so $u(k)\le 2$ by the comment
before Corollary~\ref{cor_mlocal}.  But $u(k) \ne 1$ since the form $x^2-cy^2$ is anisotropic for any non-square $c \in k$ (using $\cha k \ne 2$).  Hence $u(k)=2$.  Since these properties hold for all finite fields of characteristic not~$2$, the assertion is again a special 
case of Corollary~\ref{cor_mlocal}(\ref{mlocalcdF}).
\end{proof}

{From} Corollary~\ref{finac}, we immediately obtain the following, which in the case of $\mbb Q_p$ was recently shown by Parimala and Suresh (\cite{PaSu}, Theorem~4.6):

\begin{cor}[\label{cor_pasu}]
Let $p$ be an odd prime, and let $K$ be a finite extension of $\Q_p$ or of the field of algebraic $p$-adic numbers 
(i.e.\ the algebraic closure of $\Q$ in $\Q_p$).  If $F$ is a function field in one variable over $K$, then $u(F) = 8$.
\end{cor}

\begin{proof}
This is the case of Corollary~\ref{finac}(\ref{mlocalfinite}) with $k$ a finite field and $m=1$, taking the $1$-local field to be a $p$-adic field.
\end{proof}
Note that the above corollary shows that $u(F)\leq 8$ even if $K$ is not a finite extension but merely algebraic.

As another example of Corollary~\ref{finac}(\ref{mlocalfinite}),
let $K = \mbb Q_p((t))$ with $p$ odd, and let $F$ be a one-variable function field over $K$.  
Then $K$ is $2$-local with finite residue field, and so $u(F)=16$.

We conclude this section by proving an analog of Theorem~\ref{main} for function fields of 
patches.  This is done by means of the following lemma. We adhere to Notation~\ref{notation}.

\begin{lem}\label{reduce_to_patches}
Let $\wh X$ be a smooth connected projective curve over a complete discrete valuation ring~$T$ and let $F$ be its function field.  
Let $n \ge 0$ and assume that the
residue characteristic of $T$ does not divide $n$.  Let $U$ be a subset of the closed fiber $X$ and let $P$ be a closed point of $\wh X$.  If $a \in F_U^\times$ (resp.\ $a \in F_P^\times$) then there exists an $a' \in F$ and a unit $u \in F_U^\times$ (resp.\ $u \in F_P^\times$) such that $a=a'u^n$.\end{lem}

\begin{proof}
First consider the case that $a \in F_U^\times$.  Since $F_U$ is the fraction field of $\wh R_U$, we may write $a = a_1/a_2$ where $a_1,a_2 \in \wh R_U$ and $a_i \ne~0$.  By the 
Weierstrass Preparation Theorem for $\wh R_U$ given in 
\cite{HH:FP}, Proposition~4.7, the nonzero element $a_i \in \wh R_U$ may be
written as a product $a_i=b_ic_i$ with $b_i \in F^\times$ and $c_i \in \wh
R_U^\times$ for $i=1,2$.  Let $t$ be a uniformizer of $T$.  Then the reduction
of $c_i$ modulo $t$ is an element $\bar c_i \in \wh R_U/t \wh R_U$, the ring
of rational functions on $X$ that are regular at the points of $U$.  But this
ring is also $R_U/tR_U$.  So we may lift $\bar c_i$ to an element $c_i' \in
R_U \subset F$.  Here $c_i/c_i' \in \wh R_U^\times$, and in fact $c_i/c_i'
\equiv 1 \, {\rm mod}\, t\wh R_U$.  Now the residue characteristic of $T$
does not divide $n$, and $1$ is an $n$-th root of $c_i/c_i'$  modulo
$t$.  Hence $c_i/c_i'$ has a (non-zero) $n$-th root $c_i'' \in \wh R_U$
by Hensel's Lemma.  Thus $u := c_1''/c_2'' $ lies in $F_U^\times$, and $a' :=
b_1c_1'/b_2c_2'$ lies in $F$.  Since $a_i = b_ic_i'(c_i'')^n$, we have
$a = a_1/a_2 = a'u^n$.  This proves the result in this case.

Next consider the case that $a \in F_P^\times$.  Taking $U = \{P\}$ in the previous case, we are reduced to showing that every element $a \in F_P^\times$ is of the form $a=a'u^n$ where $a' \in F_{\{P\}}^\times$ and $u \in F_P^\times$ (because $F_{\{P\}} \subset F_P$).  By the local Weierstrass Preparation Theorem for $\wh R_P$ given in \cite{HH:FP}, 
Proposition~5.6, we may write $a_i=b_ic_i$ for some $b_i \in F_{\{P\}}^\times$ and $c_i \in \wh R_P^\times$.  
(As noted in the proof of Theorem~\ref{smooth factorization}(\ref{local}), our rings $\hat R_{\{P\}}$ and $\wh R_P$
correspond to $\wh R$ and $\wh R_1$ in \cite{HH:FP}, Section~5.)
Let $\mf m$ be the maximal ideal of $\wh R_{\{P\}}$ and let $\mf m'$ be the maximal ideal of $\wh R_P$.  So $\mf m' = \mf m \wh R_P$.  Let $\bar c_i \in \wh R_P^\times/\mf m'$ be the reduction of $c_i$ modulo $\mf m'$.  The inclusion $\wh R_{\{P\}} \hookrightarrow \wh R_P$ induces an isomorphism on the residue fields $\wh R_{\{P\}}/\mf m \wh R_{\{P\}} \to \wh R_P/\mf m' \wh R_P$; so we can regard $\bar c_i \in \wh R_{\{P\}}/\mf m \wh R_{\{P\}}$, and we can lift it to an element $c_i' \in \wh R_{\{P\}} \subset F_{\{P\}}$.  Here $c_i' \ne 0$ since $\bar c_i \ne 0$ (because $c_i \in \wh R_P^\times$).  So $c_i/c_i' \in \wh R_P^\times$ is congruent to $1$ modulo $t\wh R_P$, and so by Hensel's Lemma is an $n$-th power of some non-zero $c_i'' \in \wh R_P$.  Taking $a' = b_1c_1'/b_2c_2' \in F_{\{P\}}^\times$ and $u := c_1''/c_2'' \in \wh R_P \subset F_P$ with $u \ne 0$ then yields the desired identity $a=a'u^n$.
\end{proof}

As a consequence of this lemma and Theorem~\ref{main}, we obtain:

\begin{cor}\label{main_patches} 
Let $T$ be a complete discrete valuation ring with uniformizer $t$, whose residue field $k$ is not of
characteristic $2$.  Let $\wh X$ be a smooth projective $T$-curve with closed fiber $X$, and let $\xi$ be a proper subset of $X$
(resp.\ a closed point of $X$).  Then $4u(\kappa(Q)) \le u(F_\xi) \le 4u_s(k)$ for any closed point $Q \in X$ (resp.\ for $Q=\xi$).
\end{cor}

\begin{proof}
Let $K$ be the fraction field of $T$ and let $E, F$ be the function fields of $X, \wh X$.  Thus $F$ is a one-variable function field over $K$.  Let $k' = \kappa(Q)$, and let $I \subset \wh R_\xi$ be the ideal that defines the closed fiber $X$ locally. 

For the first inequality, consider the case when $\xi=U \subset X$.  The local ring $A$ of $X$ at $Q$ is a discrete valuation ring having residue field $k'$ and fraction field $E$.  Also, the localization of $\wh R_U$ at the prime ideal $I$ is a discrete valuation ring having residue field $E$ and fraction field $F_U$.  Applying 
Lemma~\ref{uinvdvr} to these two rings yields $u(F_U) \ge 2u(E) \ge 4u(k')$, as asserted.  In the other case, when $\xi = P \in X$ (in which case $Q=P$), if we replace the ring $A$ by its completion $\wh A$, the field $E$ by the fraction field $\wh E$ of $\wh A$, and $\wh R_U, F_U$ by $\wh R_P, F_P$, then Lemma~\ref{uinvdvr} similarly yields $u(F_P) \ge 2u(\wh E) \ge 4u(k')$.  

For the second inequality, let $q$ be a quadratic form over $F_\xi$ of dimension $n>4u_s(k)$.  We wish to show that $q$ is isotropic.  Since the characteristic of $k$ and hence of $F_\xi$ is not $2$, the form $q$ is isometric to a diagonal form $a_1x_1^2+ \cdots a_nx_n^2$ with $a_i \in F_\xi$.  By Lemma~\ref{reduce_to_patches}, $a_i = a_i'u_i^2$ for some $a_i' \in F$ and $u_i \in F_\xi^\times$.  So after rescaling $x_i$ by a factor of $u_i$, we obtain a form $q' = a_1'x_1^2+ \cdots +a_n'x_n^2$ that is isometric to $q$, with $a_i' \in F$.  The 
dimension of the $F$-form $q'$ is greater than $2u_s(K)$, since $u_s(K)=2u_s(k)$ by Theorem~\ref{main}.  Therefore $q'$ is isotropic over $F$ and hence over $F_\xi$.  Thus so is $q$.
\end{proof}

\begin{cor} \label{main_patchesacfi}
Under the hypotheses of Corollary~\ref{main_patches}, if $k$ is 
algebraically closed (resp.\ finite), then $u(F_\xi)= 4$ (resp.\ $8$).
\end{cor}

\begin{proof}  Let $k' = \kappa(Q)$.   
In the algebraically closed case the result follows from Corollary~\ref{main_patches} since $k'=k$ and $u(k)=u_s(k)=1$.  In the
finite case, $k'$ is also finite, and both $k$ and $k'$ are $C_1$-fields with $u$-invariant equal to $2$ (as noted in the proof of Corollary~\ref{finac}(\ref{mlocalfinite})).  Moreover $u_s(k)=2$ since $u(k) \le u_s(k) \le 2$ for a $C_1$-field.  So the result again follows from  Corollary~\ref{main_patches}. 
\end{proof}

For example, if $k$ is algebraically closed (resp.\ finite), then the fraction fields of $k[[x,t]]$ and $k[x][[t]]$ each have $u$-invariant equal to $4$ (resp.\ $8$).  This follows by taking $\wh X = \PP^1_{k[[t]]}$ and taking $\xi$ equal to the affine line or one point.  Similarly, taking $\wh X = \PP^1_{\Z_p}$ with $p\ne 2$, we obtain that the fraction field of $\Z_p[[x]]$ has $u$-invariant $8$, as does the fraction field of the $p$-adic completion of $\Z_p[x]$.  The above corollary can also be applied to other smooth projective curves; but by restricting attention to the line we may weaken the above hypotheses on $k$:  

\begin{cor} \label{laurent}
Let $T$ be a complete discrete valuation ring with uniformizer $t$, whose
residue field $k$ has characteristic unequal to~$2$ and satisfies $u(k)=u_s(k)$.  Then the fraction fields of $T[[x]]$ and of the $t$-adic completion of $T[x]$ have 
$u$-invariant equal to $4u(k)$.  
\end{cor}

\begin{proof}
This is immediate from Corollary~\ref{main_patches}, by taking
$\wh X = \PP^{\,1}_T$; taking $U = \A^1_T$ and $P$ to be the point $x=t=0$ in the respective cases; and taking $Q$ to be the rational point $x=t=0$ in both cases.
\end{proof}

In particular, if $k$ is any field with $u(k)=u_s(k)$, the field $k((x,t))$ has $u$-invariant equal to $4u(k)$.  
For example, if $k$ is a $C_d$-field with $u(k)=2^d$, then $k((x,t))$ has $u$-invariant equal to $2^{d+2}$.  This is because $2^d = u(k) \le u_s(k) \le 2^d$, using that $k$ is $C_d$.   
(Here, as above, we assume $\operatorname{char}(k)\neq 2$.)

%%%%%%%%%%%%%%%%%%%%%%%%%%%%%%%%%%%%%%%%%%%%%%%%%%%%%%%%%%%%%%%%%%%%%%%%%%%%%%%%
%Section 5
\section{Central simple algebras} \label{csa}
%%%%%%%%%%%%%%%%%%%%%%%%%%%%%%%%%%%%%%%%%%%%%%%%%%%%%%%%%%%%%%%%%%%%%%%%%%%%%%%%

This section contains our results on central simple algebras. 
As in the previous section, we use Theorem~\ref{torsor_injective} 
to reduce to a local problem. For basic notions concerning central simple algebras,
we refer the reader to~\cite{Sal:LN} and~\cite{Pie}. In particular, we recall that the index of a
central simple $F$-algebra~$A$ can be characterized as the degree of a minimal
\textit{splitting field} for $A$, i.e.\ a field extension $E/F$ such
that $A$ splits over $E$ in the sense that $A\otimes_F E$ is a matrix
algebra over $F$. 

The notion of a central simple algebra over a field generalizes
to that of an \textit{Azumaya algebra} over a commutative ring; see \cite{Sal:LN}, Chapter~2, or \cite{Grothbrauer}, Part~I, Section~1.  
If $A$ is an Azumaya algebra of degree $n$ over a domain
$R$, and $1 \leq i < n$, there is a functorially associated smooth
projective $R$-scheme $\SB_i(A)$, called the $i$-th \textit{generalized
Severi-Brauer variety} of $A$ (see
\cite{vdB}, p.~334, and \cite{See:BS}, Theorem~3.6; their notation is a
bit different).  For each $R$-algebra $S$, the $S$-points of $\SB_i(A)$
are in bijection with the right ideals of $A_S := A \otimes_R S$ that
are direct summands of the $S$-module $A_S$ having dimension (i.e.\
$S$-rank) $ni$.
If $R$ is a field $F$, so that $A$ is a central simple $F$-algebra,
and if $E/F$ is a field extension, then
$\SB_i(A)(E) \neq \varnothing$ if
and only if $\ind(A_E)$ divides $i$ (\cite{BofInv}, Proposition~1.17).
Here $A_E \cong \Mat_m(\Delta)$ for some $E$-division algebra $\Delta$
and some $m \ge 1$, and
the right ideals of $E$-dimension $ni$ are in natural bijection with the
subspaces of $\Delta^m$ of $\Delta$-dimension $i/\ind(A_E)$
(\cite{BofInv}, Proposition~1.12, Definition~1.9).  Thus, writing $D$
for the $F$-division algebra in the class of $A$,
the $F$-linear algebraic group $\GL_1(A) = \GL_m(D)$ acts transitively
on the points of the $F$-scheme $\SB_i(A)$ (recall the definition given
prior to Theorem~\ref{torsor_injective}).

We now place ourselves in the context of Section~\ref{rational factorization}. 

\begin{thm} \label{index hasse}
Under Notation~\ref{notation} and~\ref{P1 notation}, 
let $A$ be a central simple $F$-algebra. Then $\ind(A) = \lcm\limits_{\xi \in
\mc P \cup \mc U} \ind(A_{F_\xi})$.
\end{thm}

\begin{proof}
Let $n$ be the degree of $A$, and let $D$ be the $F$-division algebra in the class of $A$.  Then $\GL_1(A) = \GL_m(D)$ is a Zariski open subset of $\mbb A_F^{n^2}$ (because multiplication in $D$ is given by polynomials over $F$); so it is a rational connected linear algebraic group.
As noted above, if $1 \leq i < n$ then 
$\GL_1(A)$ acts transitively on the points of
$\SB_i(A)$; and if $E$ is a field extension of $F$, then $\SB_i(A)(E) \neq \varnothing$ if and only
if $\ind(A_E)$ divides $i$.  
So Theorem~\ref{torsor_injective}
implies that $\ind(A) | i$ if and only if
$\ind(A_{F_\xi}) | i$ for each $\xi \in \mc P \cup \mc U$.  Thus $\ind(A) =
\lcm\limits_{\xi \in \mc P \cup \mc U} \ind(A_{F_\xi})$ as claimed.
\end{proof}

Before proving our results about the period-index problem for central simple algebras, we recall the notion of \textit{ramification} for such algebras.
Consider an integrally closed Noetherian domain $R$ with fraction field $E$, the function field of $Y = \Spec R$.  For a codimension one irreducible subvariety 
$Z \subset Y$ with function field $\kappa(Z)$, and
an integer $n$ not divisible by the characteristic of $\kappa(Z)$, there is a canonically defined \textit{ramification map} (or \textit{residue map})
\[\oper{ram}_Z : \Br(E)[n] \to H^1(\kappa(Z), \mbb Z/n\mbb Z)\]
on the $n$-torsion part of the Brauer group (see \cite{COP}, \S2, or
\cite{Sal:LN}, pp.~67-68; here we identify $\mbb Z/n\mbb Z$ with
$\frac{1}{n}\mbb Z/\mbb Z\subseteq \mbb Q/\mbb Z$).  An element
of $H^1(\kappa(Z), \mbb Z/n\mbb Z)$ determines a cyclic Galois field extension $L/\kappa(Z)$ with a specified generator $\sigma$ of
$\Gal(L/\kappa(Z))$ whose order divides~$n$. 
For a given class $\alpha \in \Br(E)[n]$ there are only
finitely many codimension one subvarieties $Z \subset Y$ for which $\oper{ram}_Z(\alpha)$
is nonzero. We call the reduced closed subscheme supported on the union of
these varieties $Z$ the \textit{ramification divisor} of $\alpha$ (or of an algebra in its class).
By \cite{Sal:LN}, Theorem~10.3, and \cite{Grothbrauer}, Part~II, Proposition~2.3, if $R$ is regular of dimension at most~$2$ 
and $n$ is prime to the characteristics of all the residue fields
$\kappa(Z)$, then
\begin{equation}\label{star}\tag{$*$}
\xymatrix{
0 \to \Br(R)[n] \ar[r] & \Br(E)[n] \ar[rr]^-{\oplus_Z
\oper{ram}_Z} & & \bigoplus_Z H^1(\kappa(Z), \mbb Z/n\mbb Z)}
\end{equation}
is an exact sequence of abelian groups.  An $n$-torsion element of $\Br(E)$ is \textit{unramified} if its ramification divisor is trivial; i.e.\ if its image under $\oplus_Z \oper{ram}_Z$ is zero.  By the exact sequence~(\ref{star}), this is equivalent to saying that this element of $\Br(E)$ is induced by an $n$-torsion element of $\Br(R)$.

Recall (from the introduction) that we say that a field $k$ is \textit{separably closed away from $p$} if its absolute Galois group is a pro-$p$ group.  By \cite{SSS}, III.1, Proposition~16, this is equivalent to the condition that $\cd_q(k) = 0$ for all primes $q \ne p$.  By \cite{Serre:CG}, II.4.1, Proposition~11, if $q \ne \cha(k)$ the condition $\cd_q(k) = 0$ implies that $\cd_q(K) = d$ for any function field $K$ of transcendence degree $d$ over $k$.  This in turn implies that there is no non-trivial prime-to-$\cha(k)$ torsion in $\Br(K)$, for any finitely generated field $K$ over $k$ of transcendence degree $\le 1$, by applying \cite{Serre:CG}, II.2.3, Proposition~4, to such a field $K$, and using that $\Br(K) = H^2(K,{\mathbb G}_m)$.  Recall also that the \textit{Brauer dimension} of $k$ 
(resp.\ \textit{away from $p$}) is defined to be $0$ if $k$ is separably closed (away from $p$), and that otherwise it is the smallest positive integer $d$ such that for every finitely generated field extension $E/k$ of
transcendence degree $\ell \leq 1$, and every central simple $E$-algebra $A$ 
(resp.\ with $p {\not |} \per(A)$), we have $\ind(A) | \per(A)^{d+\ell - 1}$.

\begin{prop} \label{csa patches}
Let $T$ be a complete discrete valuation ring with residue field $k$, let 
$\wh X$ be a regular projective $T$-curve with function field $F$ and let $X$ be its closed fiber.    
Let $A$ be a central simple $F$-algebra whose period $n$ is not divisible by $\cha(k)$.  
Let $d \ge 0$.  Suppose that $k$ has Brauer dimension at most $d$
away from $\cha(k)$.  Under Notation~\ref{notation} we have the following:
\renewcommand{\theenumi}{\alph{enumi}}
\begin{enumerate}
\item \label{big_local_csa}
Let $X_0$ be an irreducible component of $X$.
Then 
$\ind(A_{F_U})$ divides $n^{d + 1}$ for some Zariski dense affine open subset $U \subset X_0$.
\item \label{small_local_csa}
Let $P$ be a closed point of $X$, and assume that the ramification  divisor of $A$ has at most a normal crossing singularity at $P$.  If the period $q$ of $A_{F_P}$ is a prime number unequal to  $\cha(k)$, and $F_P$ contains a primitive $q$-th root of unity, then $\ind(A_{F_P})$ divides~$q^{d+1}$.
\end{enumerate}
\end{prop}

\begin{proof}
(\ref{big_local_csa})
As in the proof of Proposition~\ref{u-inv patches}(\ref{big_local}), 
there is an affine Zariski open neighborhood $\Spec R \subset \wh X$ of the generic point of $X_0$ whose
closed fiber $U$ is an affine open subset of $X_0$ along which $X$ is regular, and such that the defining ideal of $U$ in 
$\Spec R$ is principal, say with generator $t_0 \in R \subset F$.
Let $D$ be the ramification divisor of $A$ in $\wh X$.  After shrinking $U$, 
we may assume that the ramification divisor of $A$ on $\Spec \wh R_U$ is either trivial or is the divisor of $t_0$ and that 
$\ram_U([A])$ corresponds to an \'etale
cyclic Galois cover $U' \to U$ with Galois generator $\sigma$. 
By \cite{SGA}, I, Corollaire~8.4, we may lift $U' \to U$ to obtain an \'etale 
Galois cover $\wh U' \to \Spec(\wh R_U)$, which necessarily has the same (cyclic) Galois group.  Let $\wh L/F_U$ be the
corresponding cyclic field extension and $\wh \sigma$ the lift of $\sigma$ to $\wh
L$.  Let $B$ be the cyclic $F_U$-algebra $(\wh L, \wh \sigma,
t_0)$, of degree dividing $n$ (see, for example, \cite{Sal:LN}, p.7).  
Thus $B$ is unramified away from $t_0$ on $\wh R_U$; and 
it follows from \cite{Sal:LN}, Lemma~10.2, 
that the cyclic cover of $U$ and Galois generator that are associated to $B$ agree with those associated to $A$ (i.e.\ $U'$ and $\sigma$). 
Let $C = A_{F_U} \otimes_{F_U} B^{\op}$, where $B^{\op}$ is the opposite algebra.  Notice that 
the period of $C$ divides $n$ since those of $A_{F_U}$ and $B^{\op}$ do.
Since $[B^{\op}]=-[B]$ and the ramification map is a group homomorphism, 
the central simple algebra $C$ is unramified over $\wh R_U$.

Suppose  that $\ind(C_{F_V}) | n^{d}$ for some dense open subset $V \subseteq U$. 
Since $A_{F_V}$ is Brauer equivalent to $(C \otimes_{F_U}
B)_{F_V}$, and since $\ind(B) | \deg(B)| n$,
it would then follow that $\ind(A_{F_V}) | \ind(C_{F_V}) \ind(B) |
n^{d+1}$.  So to complete the proof of (\ref{big_local_csa}) it suffices to show that $\ind(C_{F_V}) | n^{d}$ for some $V$.

Since the class of $C$ in the Brauer group is unramified over $F_U$,
the exact sequence of ramification~(\ref{star}) yields an Azumaya algebra $\mc C$ over $\wh R_U$ 
with $\per(\mc C)=\per(C)$ and 
such that $\mc C_{F_U}$ is Brauer equivalent to $C$.
Since $\per(C)$ divides $n$, the central simple algebra $\mc C_{\kappa(U)}$ has period dividing $n$ 
(here $\kappa (U)$ is the function field of $U$). By assumption on the residue field $k$, 
 $\ind(\mc C_{\kappa(U)}) | \per(\mc C_{\kappa(U)})^d|n^d=:i$ for $d>0$. In fact, the same holds if 
$d=0$ since in that case, $\per(\mc C_{\kappa(U)})=1$ by the comments before the proposition (using $\cha(k)\, {\not |}\, n$).

Let $m$ be the degree of $\mc C$ over $\wh R_U$. By tensoring $\mc C$ with a matrix algebra, we may
assume that $m > i$. We may therefore consider the $i$-th generalized Severi-Brauer $\wh R_U$-scheme $\SB_i(\mc C)$.
As noted before the statement of Theorem~\ref{index hasse}, the fact that $\ind(\mc C_{\kappa(U)}) |i$
implies the existence of a $\kappa(U)$-rational point on 
$\SB_i(\mc C_{\kappa(U)})$; or equivalently on $\SB_i(\mc C)$, by functoriality of $\SB_i$.
Hence the morphism $\pi:\SB_i(\mc C) \to \Spec \wh R_U$ has a
section $\Spec(\kappa(U)) \to \SB_i(\mc C)$ over $\Spec(\kappa(U))$,
the generic point of the closed fiber $U$ of $\Spec(\wh R_U)$. Choose a Zariski dense
open subset $V \subseteq U$ such that this section over $\Spec(\kappa(U))$ extends to a section over $V$, and such that the image of this latter section lies in an open subset of 
$\SB_i(\mc C)$ that is affine over $\wh R_U$. Then by
Lemma~\ref{affine_hensel}, the section over $V $ lifts to a section over $\Spec(\wh R_V)$.  Thus we obtain an $F_V$-point of $\SB_i(\mc C)$; or equivalently, of $\SB_i(\mc C_{F_V})$. 
Consequently, the central simple $F_V$-algebra $\mc C_{F_V}$ has index dividing
$i=n^{d}$.   But $\mc C_{F_V}$ is Brauer equivalent to $C_{F_V}$, since 
$\mc C_{F_U}$ is Brauer equivalent to $C$.  Hence $\ind(C_{F_V})$
also divides $n^d$, as desired.

(\ref{small_local_csa})
By our assumptions, $\wh R_P$ is a complete regular local ring whose fraction field contains a primitive $q$-th root of unity; 
and $A_{F_P}$ is a central simple algebra whose period is $q$ and whose ramification divisor has at most a normal crossing at $P$.  
Therefore \cite{Sal:cyclic}, Theorem~2.1, applies. In particular, $A_{F_P}$ is Brauer equivalent to $B
\otimes C$, where the class of $C$ is unramified over $\wh R_P$ and the index of $B$
divides $q^2$.  Namely, the above-mentioned theorem asserts that $B$ is either a symbol algebra of index dividing $q$ or the product of at most 
two such symbol algebras, each of which determines a cyclic extension of the residue field at a branch of the ramification divisor at $P$.
That same theorem says that the case of two symbol algebras occurs only if the cyclic 
field extension associated to one of the symbols is unramified at $P$ (and is of degree prime to $\cha(k)$).  
If $d=0$, this cyclic extension would have trivial residue field extension at $P$ by the assumption on $k$, 
and would therefore be trivial.  So in fact the index of $B$ divides $q$ if $d=0$.  That is, in general the index of $B$ divides $q^{1+e}$, 
where $e=0$ if $d=0$ and $e=1$ if $d>0$.
 
As in the proof of part (\ref{big_local_csa}),
we may find an Azumaya algebra $\mc C$ over $\wh R_P$ such that $\mc C_{F_P}$ is Brauer equivalent to $C$. By tensoring
with a matrix algebra of suitable size, we may assume that the degree
of $\mc C$ is greater than $q^{d-e}$ (with $e$ as above).  By the hypothesis on $k$, the
algebra $\mc C_k$ has index dividing $i := q^{d-e}$ (again using the comments before the proposition, in the case $d=0$, to get 
$\per(\mc C_k)=1$ and hence $\ind(\mc C_k)=1$).
Thus we obtain a section
$\Spec k \to \SB_i(\mc C)$ of $\SB_i(\mc C) \to \Spec \wh R_P$ over $\Spec k$
whose image lies in (the closed fiber of) an affine open subset of $\SB_i(\mc C)$. 
Since $\SB_i(\mc C) \to \Spec \wh R_P $ is smooth
and $\wh R_P$ is complete with residue field $k$, we may apply Lemma~\ref{affine_hensel} 
to this affine open subset and obtain a section $\Spec \wh R_P \to \SB_i(\mc C)$. This in turn gives an
$F_P$-point of $\SB_i(\mc C)$, or equivalently an $F_P$-point of
$\SB_i(\mc C_{F_P})$. In particular, we find that the index of $\mc C_{F_P}$
divides $i=q^{d-e}$. But $\mc C_{F_P}$ is Brauer equivalent to $C$. Since $A \cong B
\otimes C$ we therefore find \[\ind(A) | \ind(B) \ind(C) | q^{1+e} q^{d-e} = q^{d+1}\]
as desired.  \end{proof}

Before using the above proposition to show our main result on Brauer dimension (Theorem~\ref{main_csa}), we prove two lemmas.

\begin{lem} \label{unram}
Let $K$ be a complete discretely valued field, and 
suppose that $\alpha \in \Br(K)$ has period $n$, prime to the residue characteristic of $K$.  
Let $L$ be a totally ramified extension of $K$ of degree $n$.  
Then $\alpha_L \in \Br(L)$ is unramified.  
\end{lem}

\begin{proof}
Let $k$ be the common residue field of $K$ and $L$.
By \cite{Sal:LN}, Theorem~10.4, the ramification maps for $K$ and $L$ (with respect to the maximal ideals of the corresponding complete discrete valuation rings) form a commutative diagram
\[\xymatrix{
\Br(K) \ar[d]_{\oper{res}} \ar[r]^-{\oper{ram}} & H^1(k, \mbb Q /\mbb Z)
\ar[d]_{n} \\
\Br(L) \ar[r]^-{\oper{ram}} & H^1(k, \mbb Q/\mbb Z),
}\]
where the left hand vertical map is induced by restriction (in Galois cohomology),
and the right hand vertical map is induced by multiplication by $n$.  
Since $\alpha$ has order $n$ in the group $\Br(K)$, its image in the lower right hand $H^1(k, \mbb Q/\mbb Z)$ is zero.  
Hence $\alpha_L \in \Br(L)$ is unramified.
\end{proof}

\begin{lem} \label{resindex}
Suppose $K$ is a complete discretely valued field with residue field $k$ and
valuation ring $T$. Let $\alpha \in \Br(T)$. Then $\ind(\alpha_K) =
\ind(\alpha_k)$.
\end{lem}

\begin{proof}
Let $A$ be an Azumaya algebra in the class of $\alpha$, and let $n$ be the degree of $A$ over $T$ (which is also the degree of $A_K$ over $K$, and of $A_k$ over $k$).  For $1 \le i < n$, we have a commutative diagram of schemes
\[\xymatrix{
\SB_i(A_K) \ar[r] \ar[d]_{\pi_K} & \SB_i(A) \ar[d]_\pi & \SB_i(A_k) \ar[l]
\ar[d]_{\pi_k} \\
\Spec(K) \ar[r] & \Spec(T) & \ar[l] \Spec(k),
}\]
where $\SB_i$ is the $i$-th generalized Severi-Brauer variety.
Since $\pi$ is a proper morphism, by the valuative criterion for properness
it follows that any section of $\pi_K$ may be uniquely extended to a section of $\pi$. 
Since $\pi$ is a smooth morphism, it has a section if and only if $\pi_k$ does,
by Hensel's lemma.  This
implies that $\pi_k$ has a section if and only if $\pi_K$ has a section.  But
there is a $K$-point on $\SB_i(A_K)$ if and only if the index of $A_K$ divides $i$, and similarly for $k$.  So
$\ind(\alpha_K) | i$ if and only if $\ind(\alpha_k) | i$. Therefore $\ind(\alpha_k) =
\ind(\alpha_K)$ as desired. 
\end{proof}

\begin{thm} \label{main_csa}
Let $K$ be a complete discretely valued field whose valuation ring $T$
has residue field $k$.  
Suppose $k$ has Brauer dimension $d \ge 0$ away from $\cha(k)$. 
Then $K$ has Brauer dimension at most $d+1$ away from $\cha(k)$.
\end{thm}

\begin{proof}
Let $A$ be central simple algebra over a finitely generated field extension $F$ of $K$ having transcendence degree $\ell \le 1$, and assume that $p:=\cha(k) \ge 0$ does not divide $n := \per(A)$.  We wish to show that $\ind(A)$ divides $\per(A)^{d+\ell}$.  Let $\alpha \in \Br(F)$ be the class of $A$.

We begin by considering the case of $\ell = 0$; i.e., $F$ is a finite extension of $K$, whose residue field $k'$ is a finite extension of $k$.  If $d \ge 1$, let $L$ be a totally ramified extension of $F$ of degree $n = \per(\alpha)$.  
Thus $\alpha_L$ is unramified by Lemma~\ref{unram}.  Equivalently, by 
the exact sequence
(\ref{star}) before the statement of Proposition~\ref{csa patches}, $\alpha_L$ is induced by an element $\alpha_S$ in $\Br(S)$, where $S$ is the valuation ring of $L$.  By Lemma~\ref{resindex}, $\ind(\alpha_L) = \ind(\alpha_{k'})$, where $\alpha_{k'} \in \Br(k')$ is the class induced by $\alpha_S$. 
The hypothesis on $k$ implies that $\ind(\alpha_{k'})\, |\, \per(\alpha_{k'})^{d-1}$.  But $\ind(\alpha)\, |\, n \ind(\alpha_L)$, by \cite{Pie}, Proposition~13.4(v), since $n = [L:F]$.  Also, $\per(\alpha_{k'})\, |\, \per(\alpha_S) = \per(\alpha_L)$, since $\alpha_{k'}$ is induced by $\alpha_S$.  So $\ind(\alpha_L)\, |\,  \per(\alpha_L)^{d-1}$ and
\[\ind(\alpha)\, |\, n \ind(\alpha_L)\, |\, n \per(\alpha_L)^{d-1}\, |\, n \per(\alpha)^{d-1} = \per(\alpha)^d,\]
as desired.  

On the other hand, if $d=0$, then $k$ is separably closed away from $p$, and so has no cyclic field extensions of degree prime to $p$.  Thus $H^1(k,\mbb Z/n\mbb Z)$ is trivial and $\alpha$ is unramified.  So $\alpha$ is induced by an element $\alpha_R \in \Br(R)$, where $R$ is the valuation ring of $F$.  Let $\alpha_k$ be the induced element of $\Br(k)$.  Then $\ind(\alpha) = \ind(\alpha_k) = 1$ by Lemma~\ref{resindex} and the fact that $\Br(k)$ has no $n$-torsion (as noted before Proposition~\ref{csa patches}).  So $\ind(\alpha)\,|\,\per(\alpha)^d$ holds trivially.  This concludes the proof in the case $\ell = 0$.

We now turn to the case $\ell=1$; i.e., $F$ is a finitely generated field extension of $K$ having transcendence
degree one.  Write $n = \prod_{i=1}^m q_i^{r_i}$, where the $q_i$ are distinct primes unequal to $p$ and each $r_i \ge 1$.  Since $\alpha$ has order $n$ in the abelian group $\Br(F)$, we may write $\alpha = \alpha_1  + \cdots + \alpha_m$, where
$\alpha_i$ is $q_i$-power torsion.  Here $\per(\alpha) = \prod_i
\per(\alpha_i)$ because the $q_i$ are pairwise relatively prime.  Since the index of a tensor product of algebras
divides the product of the indices, it follows that $\ind(\alpha) | \prod_i \ind(\alpha_i)$; so without loss of generality, we may assume that $m=1$ and
that $\ind(\alpha)$ is a power of a prime $q$. Since $\per(\alpha) | \ind(\alpha)$, the period of $\alpha$ is also a power of $q$, say $n=q^r$.

Consider first the case $r = 1$, so that $\per(A) = q$. Since $\cha(F) \neq q$, the extension $F(\zeta_q)/F$, where $\zeta_q$ is
a primitive $q$-th root of unity, is an extension of $F$ of degree dividing $q-1$. Since
this is prime to $q$, we find $\ind(A) = \ind(A \otimes_{F} F(\zeta_q))$ and
$\per(A) = \per(A \otimes_{F} F(\zeta_q))$, by \cite{Pie},  Propositions~13.4(vi) and~14.4b(v).
Since $F(\zeta_q)$
is still a finitely generated extension of $K$ of transcendence degree $1$,
we may therefore assume without loss of generality that $\zeta_q \in F$. 
 
Observe (as in the proof of Theorem~\ref{main}) that there is a regular projective $T$-curve~$\wh X$ with function field $F$ such that the ramification 
divisor $D$ of $A$ on $\wh X$ has only normal crossings.  Namely,
let $\wh X_1$ be a normal projective model for $F$ over $T$, and let $D_1$ be the ramification divisor of $A$  on $\wh X_1$.  By Lemma~\ref{resolution}, there is a regular projective $T$-curve~$\wh X$ with function field $F$, and a birational morphism $\pi:\wh X \to \wh X_1$, such that $\pi^{-1}(D_1)$ has only normal crossings.  The
ramification divisor $D$ of $A$ on $\wh X$ is contained in $\pi^{-1}(D_1)$, and so it also has only normal crossings.

By Proposition~\ref{csa patches}(\ref{big_local_csa}), for each irreducible 
component $X_0$ of the closed fiber $X$ of $\wh X$, there is a Zariski dense affine open subset $U_0 \subset X_0$ such that 
$A_{F_{U_0}}$ has index dividing $q^{d+1}$. 
Let $S$ be the (finite) set of points of $X$ that do not lie in any of our chosen sets $U_0$ (as $X_0$ ranges over the components of $X$),
 together with all the closed points at which distinct components of $X$ meet.  By~\cite{HH:FP}, Proposition~6.6, there is a finite 
morphism
$f:\wh X \to \PP^1_T$ such that $S \subseteq \mc P := f^{-1}(\infty)$.  Under Notation~\ref{P1 notation}, and by the choice of $f$, 
each $U \in \mc U$
is contained in one of the above sets $U_0$; hence $F_U$ contains $F_{U_0}$.  Thus each
$A_{F_U}$ has index dividing $q^{d+1}$.  Meanwhile, since the ramification divisor of $A$ has at most normal crossings, by 
Proposition~\ref{csa patches}(\ref{small_local_csa}) we also have that the index of
$A_{F_P}$ divides $q^{d+1}$ for $P \in \mc P$. Therefore $\ind(A)$ divides $q^{d+1}$ by
Theorem~\ref{index hasse}, and the result is proven in this case.

We now consider the general case $\per(A) = q^r$ by induction on $r$.  Choose an algebra $B$ in the class $q^{r-1}[A]$.  Since $B$ has
period $q$, it has index dividing $q^{d+1}$ (by the first part of the proof for
the case $\ell=1$). Consequently, 
$B$ has a splitting field $L$ whose degree over $F$ divides $q^{d+1}$. Since $L$ is a finitely generated field extension of $K$ of
transcendence degree $1$, and $A \otimes_F L$ has period dividing $q^{r-1}$
(by definition of $L$), it
follows by induction that $A \otimes_F L$ has index dividing
$(q^{r-1})^{d+1}$. Hence
$A \otimes_F L$ has a splitting field $L'$ whose degree over $L$ divides $(q^{r-1})^{d+1}$.
Therefore $L'/F$ is a splitting field of $A$ of degree dividing
$(q^r)^{d+1}$, and the proof is complete.
\end{proof}

As in the quadratic form case, the main theorem generalizes to a result about henselian discrete valuation rings.  

\begin{cor} \label{hensel_csa}
Let $T$ be an excellent henselian discrete valuation ring having fraction field $K$ and residue field $k$.  Let $d \ge 0$.
Suppose that $k$ has Brauer dimension $d$ away from $\cha(k)$. 
Then $K$ has Brauer dimension at most $d+1$ away from $\cha(k)$.
\end{cor}

\begin{proof}
We wish to show that if $E$ is a finitely generated field extension of $K$ of transcendence degree $\ell \le 1$, and if the period of a central simple $E$-algebra $A$ is not divisible by $\cha(k)$, then $\ind(A)\,|\,\per(A)^{d+\ell-1} =:i$.  Equivalently, we wish to show that there is an $E$-point on the generalized Severi-Brauer variety $\SB_i(A)$.  

The completion $\wh T$ of $T$ is a complete discrete valuation ring with residue field $k$.  Hence by Theorem~\ref{main_csa}, the Brauer dimension of its fraction field $\wh K$ is at most $d+1$ away from $\cha(k)$.  So for every finitely generated field extension $L$ of $\wh K$ of transcendence degree $\ell$ over which $A$ is defined (e.g.\ containing $E$), the index of $A_L$ divides $\per(A_L)^{d+\ell-1}$ and hence divides $i=\per(A)^{d+\ell-1}$.  Thus $\SB_i(A)$ has a rational point over every such field $L$.  So by Lemma~\ref{approx}, 
$\SB_i(A)$ has a rational point over $E$.
\end{proof}

Recall the definition of an $m$-local field given in Section~\ref{quadratic}.

\begin{cor}\label{mlocal_csa}
Let $K$ be an $m$-local field with residue field $k$, for some $m \ge 1$.  Let $d \ge 0$, and suppose that $k$ has Brauer dimension $d$ away from $\cha(k)$. Then $K$ has Brauer dimension at most $d+m$ away from $\cha(k)$.
\end{cor}

\begin{proof} This follows from Corollary~\ref{hensel_csa} and induction.
\end{proof}

In particular, if $k$ is separably closed away from $\cha(k)$, and $F$ is a one-variable function field over an $m$-local field with residue field $k$, then $\ind(\alpha)\,|\,\per(\alpha)^m$ for any 
$\alpha \in \Br(F)$ of period not divisible by $\cha(k)$.  The above result also has the following consequence:

\begin{cor} \label{finsc_csa}
Let $K$ be an $m$-local field with residue field $k$ and let 
$F$ be a one-variable function field over $K$, where $k$ is either
\renewcommand{\theenumi}{\alph{enumi}}
\begin{enumerate}
\item \label{finitecsa}
a finite field; or
\item \label{curvecsa}
the function field of a curve over a separably closed field $k_0$.
\end{enumerate}
Then $\ind(\alpha)\,|\,\per(\alpha)^m$ (resp.\ $\ind(\alpha)\,|\,\per(\alpha)^{m+1}$) for every element in the Brauer group of $K$ (resp.\ of $F$) of period not divisible by $\cha(k)$.
\end{cor}

\begin{proof}
(\ref{finitecsa}) By Wedderburn's Theorem, $\Br(k')$ is trivial for every finite extension $k'$ of $k$.  
Moreover, period equals index in the Brauer group of any one-dimensional function field over $k$
(see \cite{Rei}, Theorem~32.19).  So the Brauer dimension of $k$ is $1$, and the conclusion follows from Corollary~\ref{mlocal_csa}.

(\ref{curvecsa}) Let $p = \cha(k_0)=\cha(k)\geq 0$.  As noted before Proposition~\ref{csa patches}, since $k_0$ is separably closed there is no non-trivial prime-to-$p$ torsion in $\Br(k)$.  Moreover, if $E$ is a one-variable function field over $k$, then $E$ is the function field of a surface over $k_0$; and hence period equals index for elements of prime-to-$p$ period in $\Br(E)$, by the main theorem of \cite{deJ}.  Thus the Brauer dimension of $k$ is $1$, and the assertion again follows from Corollary~\ref{mlocal_csa}.
\end{proof}

As an example of Corollary~\ref{finsc_csa}(\ref{curvecsa}), $\ind(\alpha)\,|\,\per(\alpha)^2$ for any element $\alpha$ in the Brauer group of $\C(x)((t))(y)$.  Also, as a special case of part~(\ref{finitecsa}) of the above result, we have the following analog of Corollary~\ref{cor_pasu} that was first proven by Saltman ~\cite{Sal:DA}:

\begin{cor} 
Let $p$ be a prime, and let $K$ be a finite extension of $\Q_p$ or of the field of algebraic $p$-adic numbers (i.e.\ the algebraic closure of $\Q$ in $\Q_p$).  If $F$ is a function field in one variable over $K$ and the period of $\alpha \in \Br(F)$ is not divisible by $p$, then $\ind(\alpha)$ divides $\per(\alpha)^2$.
\end{cor}

As another example of Corollary~\ref{finsc_csa}(\ref{finitecsa}) (taking $m=2$), the field $K = \Q_p((t))$ satisfies the relation $\ind(\alpha)\,|\,\per(\alpha)^2$ for $\alpha \in \Br(K)$ of period prime to $p$, and the field $F = \Q_p((t))(x)$ satisfies the relation $\ind(\alpha)\,|\,\per(\alpha)^3$ for $\alpha \in \Br(F)$ of period prime to $p$.

In parallel with the quadratic form situation, Theorem~\ref{main_csa} has an analog for the function fields of patches.
Namely, using Lemma~\ref{reduce_to_patches} and Theorem~\ref{main_csa} we prove

\begin{cor}\label{main_csa_patches}
Let $T$ be a complete discrete valuation ring with residue field $k$ of characteristic  $p \geq 0$. Let $\wh X$
be a smooth projective $T$-curve with closed fiber
$X$, and let $\xi$ be either a subset of 
$X$ or a closed point of $X$. Suppose that $k$ has
Brauer dimension $d$. Then for all $\alpha$ in
$\Br(F_\xi)$ with period not divisible by $p$, we have $\ind(\alpha)\, |\, \per(\alpha)^{d+2}$.  Moreover if $T$ contains a primitive $\per(\alpha)$-th root of unity, then $\ind(\alpha)\, | \,\per(\alpha)^{d+1}$.
\end{cor} 

\begin{proof} 
As in the proof of Theorem~\ref{main_csa} in the case $\ell=1$, by considering the prime factorization of $\per(\alpha)$ we reduce to the case that $\per(\alpha)$ is a prime power, say $q^r$.  

Let $T' = T[\zeta_{q^r}]$, where $\zeta_e$ denotes a primitive $e$-th root of unity.  Let $\wh X' = \wh X \times_T T'$, with function field $F' = FT'$, and let $\xi' = \xi \times_T T'$ in $\wh X'$.  Then $F_{\xi'} = F_\xi T' = F_\xi(\zeta_{q^r})$, where $F_{\xi'}$ is as in Notation~\ref{notation} with respect to the curve $\wh X'$.  Consider the intermediate field $F_\xi(\zeta_q)$.
The degree $[F_\xi(\zeta_q):F_\xi]$ divides $q-1$ and $s := [F_{\xi'}:F_\xi(\zeta_q)]$ divides $q^{r-1}$.  Let $\alpha' \in \Br(F_{\xi'})$ and $\alpha'' \in \Br(F_\xi(\zeta_q))$ be the elements induced by $\alpha \in \Br(F_\xi)$.
Since $[F_\xi(\zeta_q):F_\xi]$ is prime to the period of $\alpha$, the period and index of $\alpha''$ are equal to those of $\alpha$ 
(\cite{Pie},  Propositions~13.4(vi) and~14.4b(v)).
By \cite{Pie}, Proposition~13.4(v), $s \ind(\alpha')$ is divisible by $\ind(\alpha'') = \ind(\alpha)$.    

Since $F_{\xi'}$ contains $\zeta_{q^r}$, by \cite{MerSus} the element $\alpha' \in \Br(F_{\xi'})$ is represented by a tensor product  
$(a_1, b_1)_{q^r} \otimes \cdots \otimes~(a_m, b_m)_{q^r}$
of symbol algebras, 
where each $a_i,b_i \in F_{\xi'}$. Applying Lemma~\ref{reduce_to_patches} to the smooth projective $T'$-curve $\wh X'$, we may write $a_i = a_i' u_i^{q^r}$ and $b_i = b_i' v_i^{q^r}$ for $a_i', b_i' \in F'$ and $u_i, v_i
\in F_{\xi'}$.  Thus 
$(a_i, b_i)_{q^r}$ is Brauer equivalent to $(a_i', b_i')_{q^r}$. 
So if we consider the central simple $F'$-algebra $A = (a_1', b_1')_{q^r} \otimes \cdots \otimes
(a_m', b_m')_{q^r}$, then the class of $A \otimes_{F'} F_{\xi'}$ is $\alpha'$.  By Theorem~\ref{main_csa},  $\ind(A)\, |\, \per(A)^{d+1}$.  But $\ind(\alpha)$ divides $s \ind(\alpha')$ and hence $s\ind(A)$; and $\per(A)$ divides $q^r = \per(\alpha)$, since $\per(a_i,b_i)_{q^r}\,|\,q^r$.  So $\ind(\alpha)\,|\,s\per(\alpha)^{d+1}$.

Since $s$ divides $q^{r-1}$, this shows that $\ind(\alpha)\,|\,
\per(\alpha)^{d+2}$.  In the case that $T$ (and hence $F_\xi$) contains a primitive $q^r$-th root of unity, $s=1$ and so $\ind(\alpha)\,|\,\per(\alpha)^{d+1}$.
\end{proof}

\begin{cor} \label{main_csa_patchessc}
Under the hypotheses of Corollary~\ref{main_csa_patches}, if $k$ is  separably closed, then $\per(\alpha)=\ind(\alpha)$ for elements 
in $\Br(F_\xi)$ of period not divisible by the characteristic of $k$.
\end{cor}

\begin{proof}
Since the characteristic of $k$ does not divide $\per(\alpha)$, it follows that $k$ contains a primitive $\per(\alpha)$-th root of unity.  
Moreover $k$ has Brauer dimension zero.  So $\ind(\alpha)$ divides $\per(\alpha)$ by Corollary~\ref{main_csa_patches}.  But $\per(\alpha)$ divides $\ind(\alpha)$; so the result follows.
\end{proof}

In particular, if $k$ is separably closed, then period equals index for elements of period not divisible by $\cha(k)$ in the Brauer groups of the fraction fields of $k[[x,t]]$ and $k[x][[t]]$.  Similarly, let 
$\Z_p^{\ur}$ be the maximal unramified extension of $\Z_p$.  The residue field of $\Z_p^{\ur}$ is the algebraically closed field $\bar \F_p$, and so the fraction fields of 
$\Z_p^{\ur}[[x]]$ and of the $p$-adic completion of $\Z_p^{\ur}[x]$ each have the property that period equals index for elements in their Brauer group having period prime to $p$.

\begin{rem} \label{csa_patchrk}

\ \ \ (a) The proof of Corollary~\ref{main_csa_patches} actually shows more: that the index of $\alpha \in \Br(F_\xi)$ divides $[F_\xi(\zeta_n):F_\xi(\zeta_{\rho(n)})]n^{d+1}$, where $n= \per(\alpha)$ and where $\rho(n)$ denotes the product of the distinct primes that divide $n$ (each taken with multiplicity one).  In particular, the index of $\alpha$ in $\Br(F_\xi)$ divides $\per(\alpha)^{d+2}/\rho(\per(\alpha))$.

(b) We suspect that actually $\ind(\alpha) | \per(\alpha)^{d+1}$ in Corollary~\ref{main_csa_patches}, even without the assumption on roots of unity.  Perhaps this could be shown by paralleling the proof of Theorem~\ref{main_csa} with $F$ replaced by $F_\xi$.  But doing this would require generalizations of previous results here and in \cite{HH:FP}.
\end{rem}

Remark~\ref{csa_patchrk}(a) shows that 
Corollary~\ref{main_csa_patchessc} can be strengthened to include the 
case that $k$ is separably closed away from $p = \cha(k)$.  To see this, 
first note for any integer $n$, the degree
$[F_\xi(\zeta_n):F_\xi(\zeta_{\rho(n)})]$ is divisible only by primes 
that divide $n$.  Now let $\alpha$ be an element of $\Br(F_\xi)$ whose 
period $n$ is not divisible by $p$.  Then the above degree is prime to 
$p$.  But $k$ is separably closed away from $p$.  So in fact this degree 
is equal to $1$.  Since the Brauer dimension $d$ of $k$ is zero, 
Remark~\ref{csa_patchrk}(a) then shows that $\ind(\alpha)$ divides (and 
hence is equal to) $n = \per(\alpha)$.

%%%%%%%%%%%%%%%%%%%%%%%%%%%%%%%%%%%%%%%%%%%%%%%%%%%%%%%%%%%%%%%%%%%%%%%%%%%%%%%%
%%%%%%%%%%%%%%%%%%%%%%%%%%%%%%%%%%%%%%%%%%%%%%%%%%%%%%%%%%%%%%%%%%%%%%%%%%%%%%%%
\bibliographystyle{alpha}
\bibliography{citations}
\def\cprime{$'$} \def\cprime{$'$} \def\cprime{$'$} \def\cprime{$'$}
  \def\cftil#1{\ifmmode\setbox7\hbox{$\accent"5E#1$}\else
  \setbox7\hbox{\accent"5E#1}\penalty 10000\relax\fi\raise 1\ht7
  \hbox{\lower1.15ex\hbox to 1\wd7{\hss\accent"7E\hss}}\penalty 10000
  \hskip-1\wd7\penalty 10000\box7}

\medskip

\noindent Author information:

\medskip

\noindent David Harbater: Department of Mathematics, University of Pennsylvania, Philadelphia, PA 19104-6395, USA\\ email: {\tt harbater@math.upenn.edu}

\medskip

\noindent Julia Hartmann: Lehrstuhl A f\"ur Mathematik, RWTH Aachen University, 52062 Aachen, Germany\\ email:  {\tt 
Hartmann@mathA.rwth-aachen.de}

\medskip

\noindent Daniel Krashen: Department of Mathematics, University of Georgia, Athens, GA 30602, USA\\ email: {\tt dkrashen@math.uga.edu}

\end{document}